\documentclass[12pt,a4paper]{article}
\usepackage[utf8]{inputenc}
\usepackage[OT1]{fontenc}
\usepackage{amsmath}
\usepackage{amsthm}
\usepackage{amsfonts}
\usepackage{amssymb}
\usepackage{dsfont}
\usepackage{graphicx}
\usepackage{cite}
\usepackage{hyperref}
\usepackage{tikz-cd}
\usepackage{xcolor}
\usepackage{tikz}
\usepackage[left=2cm,right=2cm,top=2cm,bottom=2cm]{geometry}
\graphicspath{{pic/}}

\newtheorem{proposition}{Proposition}[section]
\newtheorem{lemma}{Lemma}[section]
\newtheorem{theorem}{Theorem}[section]
\newtheorem{corollary}{Corollary}[section]

\theoremstyle{definition}
\newtheorem{definition}{Definition}[section]
\newtheorem{remark}{Remark}

\newtheorem{example}{Example}[section]

\newtheorem*{solution*}{Solution}

\DeclareMathOperator{\supp}{supp}

\DeclareMathOperator{\dive}{div}

\DeclareMathOperator{\ran}{Ran}

 \newcommand{\bbar}[1]{\setbox0=\hbox{$#1$}\dimen0=.2\ht0 \kern\dimen0 \overline{\kern-\dimen0 #1}}

 \DeclareMathOperator{\End}{\ensuremath{\mathcal{E}\kern-.125em\mathpzc{nd}}}

 \DeclareMathOperator{\Hom}{\mathcal{H}\kern-.125em\mathpzc{om}}
 
 \DeclareMathOperator{\id}{id}

 \DeclareMathOperator{\Proj}{\mathcal{P}\kern-.125em\mathpzc{roj}}

 \renewcommand{\setminus}{\smallsetminus}

 \newcommand{\udot}{\ensuremath{{\lower .183333em \hbox{\LARGE \kern -.05em$\cdot$}}}}

 \newcommand{\cA}{\mathcal{A}}

 \newcommand{\cD}{\mathcal{D}}
 
 \newcommand{\cF}{\mathcal{F}}
 \newcommand{\cG}{\mathcal{G}}
 \newcommand{\cH}{\mathcal{H}}

 \newcommand{\cK}{\mathcal{K}}
 \newcommand{\cL}{\mathcal{L}}

\newcommand{\cS}{\mathcal{S}}

 \newcommand{\cV}{\mathcal{V}}
 
 \newcommand{\cX}{\mathcal{X}}
 
 \newcommand{\cZ}{\mathcal{Z}}

 \newcommand{\N}{\mathbb{N}}

 \newcommand{\R}{\mathbb{R}}

 \newcommand{\U}{\mathbb{U}}

  \newcommand{\Z}{\mathbb{Z}}

 \newcommand{\fg}{\mathfrak{g}}

 \newcommand{\p}{\partial}

 \DeclareMathOperator{\rank}{rank}
 \DeclareMathOperator{\spann}{span}

 \DeclareMathOperator{\Prim}{Prim}  
 \DeclareMathOperator{\Lie}{Lie} 
 \DeclareMathOperator{\Diff}{ Diff}

\author{Ivan Beschastnyi}

\title{Closure of the Laplace-Beltrami operator on 2D almost-Riemannian manifolds and semi-Fredholm properties of differential operators on Lie manifolds}

\begin{document}

\maketitle

\abstract{The problem of determining the domain of the closure of the Laplace-Beltrami operator on a 2D almost-Riemannian manifold is considered. Using tools from theory of Lie groupoids natural domains of perturbations of the Laplace-Beltrami operator are found. The main novelty is that the presented method allows us to treat geometries with tangency points. This kind of singularity is difficult to treat since those points do not have a tubular neighbourhood compatible with the almost-Riemannian metric.}

\section{Introduction}

Singular differential equations are ubiquitous in mathematics and physics. Some of these singularities arise from singular changes of variables such as the passage from Euclidean to spherical coordinates. Some singularities are structural and constitute an important part of a theory. A good example is the electrostatic potential often called the Coulomb potential which is equal to one over the distance to a chosen point. No matter what is the reason for the singularity, such structures are called singular in a contrast to regular structures. Very often this is just an indication that an important standard tool, whether this is the implicit function theorem or elliptic estimates, which worked perfectly well in the regular class of objects is no longer available for the new class of objects we wish to study.

Sub-Riemannian geometry from this point of view can be seen as a singular analogue of Riemannian geometry. Indeed, there are many Riemannian results that find analogues in the sub-Riemannian realm. Just to name a few: sub-elliptic estimates~\cite{stein_roth,bramanti}, Laplacian comparison theorems~\cite{lap_comp}, Brunn-Minkowski inequalities~\cite{davide_luca}, various Hardy-type inequalities~\cite{vale_hardy}, Weyl laws~\cite{weyl_sr} and others. However those results are not merely copies of proofs from Riemannian geometry. Sub-Riemannian geometry had to reinvent many tools from scratch and rethink the notion of geodesics, curvature, Jacobi fields or volume so that those concepts would apply to both Riemannian and sub-Riemannian spaces. This extended the class of objects that we would consider regular.

The modern definition of a sub-Riemannian manifold is very general~\cite{abb}. We will give it later in Section~\ref{sec:ar_geom}. If we want to capture it essence, it is already possible to do it locally. Let $X_1,\dots,X_k\in \Gamma(TM)$ be some vector fields. We define $\cD$ to be the $C^\infty(M)$-module generated by $X_1,\dots,X_k$, which can be seen as a possibly rank-varying distribution of planes 
$$
\cD_q = \spann\{X_1(q),\dots X_k(q)\}.
$$ 
Moreover we declare $X_1,\dots,X_k$ to be orthonormal. At this moment it is not entirely clear what orthonormality means at the points where $X_i$, $i\in\{1,\dots,k\}$ are linearly dependent and this will be clarified once the global definition is stated. In particular, some of $X_i$, $i\in\{1,\dots,k\}$ can vanish. If for a generic point $q$ of a sub-Riemannian manifold $M$ we have $\rank \cD_q = \dim M $, then we call such a structure \textit{almost-Riemannian}. In this case the set of points $q\in M$ where $\rank \cD_q < \dim M $ is called \textit{the singular set} and we will denote it by $\cZ$.

Almost-Riemannian structures were extensively studied in~\cite{ar1,ar2,ar3,ar4,ar5,ar6,ar7,grushin}. Unlike sub-Riemannian manifolds they are equipped with an array of canonical Riemannian objects, such as a curvature or volume. However all of those quantities explode at the singular set $\cZ$. In particular, an almost-Riemannian manifold has an infinite Riemannian volume. Thus if the singular set $\cZ$ is a smooth embedded submanifold of $M$, we can look at the almost-Riemannian manifold $M$ in two different ways: as the original complete sub-Riemannian manifold or as a non-complete open Riemannian manifold $M\setminus \cZ$. 

These two points of view are in sharp contrast with each other. For example, if on a 2D almost-Riemannian manifold the Gaussian curvature is negative everywhere where it is defined, the almost-Riemannian geodesics can still converge~\cite{so_me}. Or if we look at the corresponding Laplace-Beltrami operator $\Delta$ with domain $D(\Delta)$ being equal to $C^\infty_c(M\setminus \cZ)$ the space of smooth functions with compact support outside the singular set, then this operator is essentially self-adjoint. This implies that geodesics can easily cross the singular set, while a quantum particle or heat flow can not~\cite{grushin}. This phenomena is now known in the literature as quantum confinement.

If we want to understand sub-Riemannian structures at the level of generality of Definition~\ref{def:sub} we need first to understand the simplest case of 2D almost-Riemannian manifolds. Those structures were locally classified in~\cite{ar3}. The authors of that article proved that structurally there are three generic types of local behaviour:
\begin{enumerate}
\item There might be no singularity at $q\in M$, i.e., $q\notin \cZ$. We call such a point a \textit{Riemannian point}. A good example of a space without any singular points is just the Euclidean plane which can be seen as an almost-Riemannian structure generated by two vector fields
$$
X_1 = \p_x, \qquad X_2 = \p_y;
$$
\item If $q\in \cZ$ and $\dim \Delta_q = 1$ with $\Delta_q$ transversal to the singular set $\cZ$ we call $q\in \cZ$ a \textit{Grushin point}. A good example of an almost-Riemannian structure with only Grushin and Riemannian points is given by the Grushin plane generated by two vector fields on $\R^2$:
$$
X_2 = \p_x, \qquad X_2 = x\p_y;
$$ 
\item If $q\in \cZ$ and $\dim \Delta_q = 1$ with $\Delta_q$ tangent to the singular set $\cZ$ we call $q\in \cZ$ a \textit{tangency point}. Two give an example consider a structure on $\R^2$ generated by two vector fields:
$$
X_2 = \p_x, \qquad X_2 = (y-x^2)\p_y;
$$ 

\end{enumerate}
The last example contains all three types of points and is depicted in Figure~\ref{fig:class}.

\begin{figure}
\begin{center}
\includegraphics[scale=0.7]{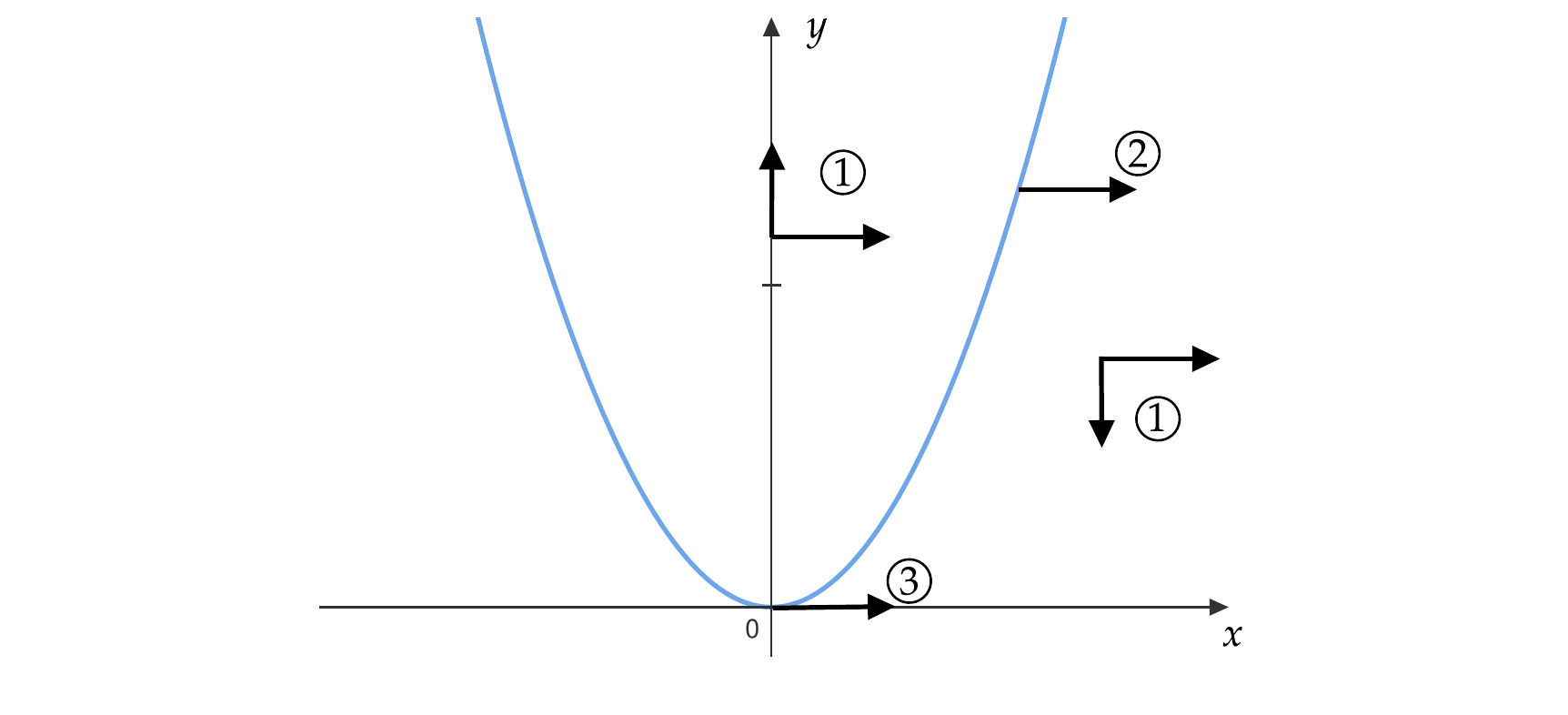}
\caption{Three generic types of points on a 2D almost-Riemannian manifold with singular set given by the parabola $y=x^2$: 1) Riemannian points, 2) Grushin points, 3) tangency points.\label{fig:class}}
\end{center}
\end{figure}

Let us write down the Laplace-Beltrami operators for each structure in order to see concretely what kind of singularities occur. We have
\begin{align*}
\Delta_{Euclidean} & = \p_x^2 + \p_y^2;\\
\Delta_{Grushin} &= \p_x^2 + x^2 \p_y^2 - \frac{1}{x}\p_x; \\
\Delta_{Tangency} &= \p_x^2 + (y-x^2)^2 \p_y^2 + \frac{2x}{y-x^2}\p_x + (y-x^2)\p_y.
\end{align*}
One can ask many natural questions about these operators. What are their natural domains? What can we tell about their spectrum? For which classes of functions we can solve the Poisson equation? And many others. However when we try to answer these questions we find that techniques that work well for Riemannian structures often do not work in the presence of Grushin points. Similarly the techniques that work for Grushin structures often do not work in the presence of the tangency points. In an excellent series of papers~\cite{luca1,luca2,luca3} the authors derived very general conditions for quantum confinement. However those conditions rely on an assumption that the distance from the singular set is at least $C^2$. While this assumption indeed holds for Grushin structures, it is never true for structures with tangency points as the authors themselves point out in~\cite{luca3}.

In this paper we study the closure of the Laplace-Beltrami on a generic compact orientable 2D almost-Riemannian manifold without a boundary. We wish to find a result that would hold for all generic structures no matter the singularity. We will only consider the singular cases, since the non-singular Riemannian case is covered by the standard elliptic theory. The method that we employ here for certain reasons does not apply to the Laplace-Beltrami operator itself. It works however for certain perturbations of the operator and for some non-generic structures. In the non-perturbed generic case it is still possible to extract some useful information about the closure as we will see in a simple model example.

If we wish to deal with singular objects the first thing one might want to do is to look at how singularities are treated in other branches of mathematics. If we wish to stay in the realm of differential geometry, then Poisson geometry is a good example. Indeed, regular Poisson structures are symplectic manifolds and for them the Poisson tensor is non-degenerate. Singular Poisson structures however arise very naturally in the study of Lie groups as Lie-Poisson structures of duals of Lie algebras. These singularities in some sense are of a similar nature to the sub-Riemannian singularities as in both cases they arise from some rank dropping conditions. In order to study all the variety of singular phenomena it is common to use \textit{Lie groupoids} and \textit{Lie algebroids} which became a standard language of Poisson geometry~\cite{poisson}. 

Lie groupoids have many faces and definitions. In Section~\ref{sec:lie} we will give a purely differential geometric definition. They can be seen as objects that interpolate between manifolds and Lie groups or as manifolds with partial symmetries. However for our purposes they will serve as natural desingularisations of singular spaces. Besides Poisson geometry they indeed are used in this fashion when studying orbifolds~\cite{orbifolds}, foliations~\cite{folliations} or just manifolds with boundary~\cite{boundary}. In the recent years it was understood that Lie groupoids constitute a good class of spaces to do analysis on. One can construct, for example, pseudo-differential calculus~\cite{vanerp} and Fourier integral operators~\cite{lescure} adapted to their algebraic structure, use them to prove index theorems~\cite{connes} or study the essential spectrum of differential operators~\cite{nistor_fred}.

For our problem of finding the closure of the Laplace-Beltrami operator we will only need a special class of Lie groupoids that come from Lie manifolds introduced in~\cite{sobolev}. A Lie manifold is a pair $(M,\cV)$, where $M$ is a smooth manifold with boundary and $\cV$ is a Lie subalgbera of the Lie algebra of vector fields tangent to the boundary $\p M$ satisfying certain conditions. In particular, in the interior of $M$ set $\cV$ coincides with the Lie algebra of vector fields. The difference arises from the behaviour at the boundary. To each Lie manifold it is possible to associate compatible structures: metric $g_\cV$, volume $\mu_\cV$ and Sobolev spaces $H^k_{\cV}(M)$.

We will endow every compact almost-Riemannian manifold with a Lie manifold structure. For this purpose we will cut the almost-Riemannian manifold along the connected components and view it as an interior of a Riemannian manifold with boundary. To each boundary component $\cZ_i$ (or their union) we can then associate a \textit{defining function}, which is a bounded function $s\geq 0$ such that $s(q) =0$ if, and only if, $q\in \cZ_i$ and $ds|_{\cZ_i} \neq 0$. We are now ready to state the result about the closure of the Laplace-Beltrami operator.

\begin{figure}[ht]
\begin{center}
\includegraphics[scale=1]{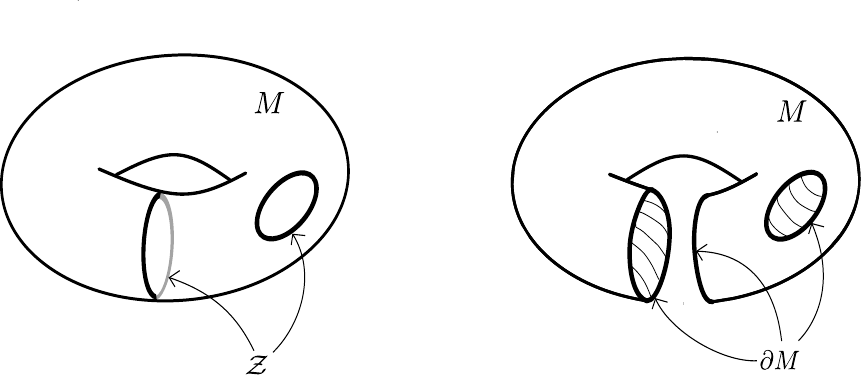}
\caption{An almost-Riemannian manifold $M$ before and after the cut along the singular set $\cZ$.\label{fig:cut}}
\end{center}
\end{figure} 

\begin{theorem}
\label{thm:main}
Consider a compact generic 2D almost-Riemannian manifold $(M,\U,f)$ and associate to it a Riemannian manifold by cutting $M$ along the singular set $\cZ$ and take a connected component, which we denote via $M$ by abusing the notation. Let $\omega$ be the Riemannian volume form, $\Delta$ the associated Laplace-Beltrami operator and $\cV$ the associated Lie manifold structure. Suppose that $h\in C^\infty(M)$ is a strictly positive on $\cZ$ function and let $s$ be a defining function of the singular set $\cZ$. Define
$$
\tilde{\Delta} = \Delta - \frac{h}{s^2}
$$
with domain $D(\tilde{\Delta}) = C^\infty(M\setminus \cZ)$.

Then the domain of closure of $\tilde{\Delta}$  in $L^2(M,\mu)$
$$
D (\overline{\tilde{\Delta}} ) = s H^2_{\cV}(M).
$$
If there are no tangency points, then $h$ can be taken just non-vanishing on $\cZ$.
\end{theorem}

The proof will rely on the following Proposition proved in~\cite{mendoza}.
\begin{proposition}
\label{prop:mendoza}
Let $H_1, H_2$ be Hilbert spaces, and let $D$ be a Banach subspace of $H_1$ equipped with a norm $\|\cdot\|_D$ such that the inclusion map $(D,\|\cdot\|_D)\to (H_1,\|\cdot\|_{H_1})$ is continuous and assume that:
\begin{enumerate}
\item The range of $A$ is closed.
\item The kernel $\ker A \subset H_1$ is closed with respect to $\|\cdot\|_{H_1}$.
\end{enumerate}
Then the operator $A$ with domain $D$ is closed, i.e., $D$ is complete with respect to the graph norm $\|u\|_{A} = \|u\|_{H_1} + \|Au\|_{H_2}$.
\end{proposition}

Note that in particular $A$ can be left semi-Fredholm, which means that $\dim \ker A< +\infty$ and $\ran A$ is closed. The strategy of the proof of Theorem~\ref{thm:main} will consist of the following steps:
\begin{enumerate}
\item Prove that functions $C^\infty_c( M)$ are dense in $s H^2_{\cV}(M)$;
\item Prove that $s H^2_{\cV}(M)$ is continuously embedded into $L^2(M,\omega)$;
\item Prove that $\tilde{\Delta}:s H^2_{\cV}(M) \to L^2(M,\omega)$ is left semi-Fredholm.
\end{enumerate}

The first two points will essentially from the definitions (see also~\cite{sobolev} the general theory of Sobolev spaces on Lie manifolds). The main difficulty is the third point. A classical way to prove that an operator is left semi-Fredholm is to construct a left parametrix. This is indeed exactly what authors of~\cite{mendoza} did. However they construct a pseudodifferential calculus adapted to their problem and following this approach each time is quite difficult. In~\cite{nistor_fred} Nistor, Carvalho and Qiao provide an alternative approach for determining whether a pseudodifferential operator (PDO) on an open manifold $M$ is Fredholm. More precisely, they prove the following result.

\begin{theorem}[Carvalho-Nistor-Qiao (CNQ) conditions]
Let $P$ be an order $m$ classical PDO on an open manifold $M$ compatible with the geometry. Then one can associate to $P$ the following data:
\begin{enumerate}
\item Smooth manifolds $M_\alpha$, parametrised by a suitable set $I$;
\item Simply connected Lie groups $G_\alpha$ acting freely and properly on $M_\alpha$, $\alpha \in I$;
\item Limit operators $P_\alpha$, which are $G_\alpha$-invariant pseudodifferential operators on $M_\alpha$;
\end{enumerate}
and the following statement holds:
$$
P: H^s(M) \to H^{s-m}(M) \text{ is Fredholm} \iff P \text{ is elliptic and}
$$
$$
P_\alpha: H^s(M_\alpha) \to H^{s-m}(M_\alpha) \text{ are invertible for every } \alpha \in I
$$
for some suitable Sobolev spaces $H^k$.
\end{theorem}

The statement will be made more precise after all the relevant notions are introduced. We will see that a similar statement holds for left semi-Fredholm operators. Namely if $P$ is elliptic and all $P_\alpha$ are \textit{left invertible}, then $P$ is \textit{left semi-Fredholm}. It should be noted that there exists extensive literature on Fredholm conditions in a great variety of situations (see, for example, \cite{georgescu1,georgescu2,mantoiu1,mantoiu2,mantoiu3}).

In the case of almost-Riemannian manifolds $I$ is the singular set $\cZ$ and $M_\alpha = G_\alpha$. Thus $P_\alpha$ will be some right invariant operators and the full power of non-commutative harmonic analysis can be exploited. However in the low-dimensional case this is not strictly speaking necessary. Indeed, it is known that there are only two non-isomorphic simply connected Lie groups in dimension two: the Euclidean space and the group of affine transformations of the real line. We will see that $G_q$ is Euclidean if, and only if, $q$ is a tangency point. This means surprisingly that it is much easier to verify the left invertibility of a  limit operator at a tangency point than at a Grushin point. Indeed, the main work will consists of dealing with Grushin points.
It is worth emphasising again that using the techniques explained in the article one can also obtain some information about the closure of the non-perturbed operator as well. Theorem~\ref{thm:main} is stated mainly to show that various singularities can be treated in a unified manner. Further generalisations can be proved similarly. Besides this, it seems there are not many results on the Laplace-Beltrami operator in the presence of tangency points. At least the author is not aware of any result in this direction, though there are several works which study structures without tangency points~\cite{grushin,luca1,luca2,luca3,EU,alessandro} including a work by the author, Ugo Boscain and Eugenio Pozzoli~\cite{me}. Indeed the main motivation for this work came from the attempt to find self-adjoint extensions of the curvature Laplacian on Grushin manifolds similarly to what was done in~\cite{EU,alessandro} and explore their unusual quantum mechanical behaviour. If in the construction of the Laplace-Beltrami operator instead of the Riemannian volume we take any smooth volume, then there is no singularity present (except the degeneracy of the principal symbol) and such operators can be handled using essentially the theory of H\"ormander operators even though the singularity still manifests itself in various forms~\cite{trelat_heat,luca4}. We should also mention separately article~\cite{heat_tommaso} where some results for the heat content on domains with characteristic points were obtained. Finally, we note that the use of Lie groupoids in sub-Riemannian geometry is not new. Indeed, they were used in articles~\cite{vanerp2,ponge,dave} where authors mainly focus on equiregular structures and are aimed towards index theory on filtered manifolds. This theory was later extended to general Kolmogorov type operators in~\cite{yuncken2}.

We end this rather lengthy introduction by explaining the structure of the paper. This paper is written for two separate communities. On one hand for the sub-Riemannian community, on the other for people working on analysis and $C^*$-algebras on Lie groupoids. For this reason considerable part of the article is an explanation of basic notions both from almost-Riemannian geometry and Lie theory with many pictures, so that both communities can understand the idea behind definitions and methods used in the paper. In Section~\ref{sec:ar_geom} basic notions from sub-Riemannian and almost-Riemannian geometry are given. In Section~\ref{sec:lie} the basics of Lie groupoids, Lie algebroids and Lie manifolds with relevant examples are discussed. In Section~\ref{sec:anal} the compatible PDO calculus and Sobolev spaces are defined. In Section~\ref{sec:cnq} generalised Carvalho-Nistor-Qiao conditions are stated and an idea of the proof is given. In Section~\ref{sec:1D} we give a simple application of the Carvalho-Nistor-Qiao conditions to the study of a model 1D example that will be useful in the study of limit operators at Grushin points. Section~\ref{sec:ar_anal} is entirely dedicated to the proof of Theorem~\ref{thm:main}. The goal of this article is to show that it is possible to study different singularities in sub-Riemannian geometry via a unique single theory. For this reason we do not strive for the most general results, nevertheless in the final Section~\ref{sec:final} we discuss the applicability of this method to other structures. The reader familiar with Lie groupoids and their representations can simply skim Sections~\ref{sec:lie}-\ref{sec:anal} to become familiar with the notations of the article and go directly to the proof of Theorem~\ref{thm:main}. For people familiar with sub-Riemannian geometry a short summary on the limit operators and how to write them explicitly is given in the beginning of Section~\ref{sec:1D}, where we consider a relevant 1D example. Sections~\ref{sec:lie}-\ref{sec:anal} essentially collect all the definitions and necessary for this paper results which are scattered over the literature. The author hope that both communities will find the results interesting either in terms of techniques, or examples and possible applications.

\section{Almost-Riemannian geometry}

\label{sec:ar_geom}

Let us recall the definition of a sub-Riemannian manifold from~\cite{abb}. Let $\cF$ be a set of vector fields on a manifold $M$. The Lie algebra generated by $\cF$ is defined as the smallest Lie subalgebra of the Lie algebra of vector fields on $M$ containing all the commutators of $\cF$:
$$
\Lie \cF = \spann\{[X_1,\dots,[X_{j-1},X_j]], X_i\in \Gamma(\cF), j\in \N\}.
$$ 
We denote by 
$$
\Lie_q \cF = \spann\{X(q)\, : \, X \in \cF\}
$$ 
and say that $\cF$ satisfies the \textit{H\"ormander condition} or is \textit{bracket-generating} if
$$
\Lie_q \cD = \{X(q):X \in \Lie \cD \} = T_q M, \qquad \forall \in M.
$$

\begin{definition}
\label{def:sub}
Let $M$ be a connected smooth manifold. A \textit{sub-Riemannian structure} on $M$ is a pair $(\U,f)$, where 
\begin{enumerate}
\item $\pi_\U: \U\to M$ is a Euclidean bundle;
\item $f:\U \to TM$ is a fiber-wise linear smooth morphism of bundles. In particular the following diagram is commutative
$$
\begin{tikzcd}
\U \arrow[dr, "\pi_\U" ']\arrow[r, "f"] & TM \arrow[d, "\pi"] \\& M
\end{tikzcd}
$$
\item The family of vector fields $\cD = f(\Gamma(\U))$ satisfies the H\"ormander condition.
\end{enumerate}
\end{definition}

The family $\cD$ is called the \textit{distribution} of the sub-Riemannian structure $(\U,f)$. Denote $\cD_q = f(\U_q)$. A Lipschitz curve $\gamma: [0,1]\to M$ is called admissible if for almost every $t\in [0,1]$ the velocity vector $\dot{\gamma}(t) \in D_{\gamma(t)}$.

\begin{theorem}[Rashevsky-Chow]
If a sub-Riemannian manifold $(M,\U,f) $ satisfies the H\"ormander condition, then any two points of $M$ can be connected by an admissible curve.
\end{theorem}

We can measure the lengths of admissible curves. Indeed, the sub-Riemannian norm of  a vector $v\in \cD_q$ is defined as 
$$
\|v\| = \min \{|u|\, : \, u\in \U_q,\, v= f(q,u) \}.
$$
The the length of an admissible curve $\gamma: [0,1]\to M$ is defined as
$$
l(\gamma) = \int_0^1 \|\dot{\gamma}\| dt.
$$
H\"ormander condition under some additional completeness assumptions guarantees that any two points can be connected by a minimiser. In this case a sub-Riemannian manifold becomes a well-defined metric space.

We define the flag $\cD^1_q \subset \cD^2_q \subset \dots \subset T_q M$ of a sub-Riemannian manifold at a point $q$ recursively as
$$
\cD^1_q = \cD_q, \qquad \cD^{k+1} = \cD^k_q + [\cD^{k},\cD]_q.
$$
H\"ormander condition equivalently says that for each point $q\in M$ there exists a number $k(q)\in \N$ such that $\cD^{k(q)}_q = T_q M$.

\begin{definition}
An \textit{almost-Riemannian manifold} (AR) is a sub-Riemannian manifold where $\dim \cD_q = \dim M$ or equivalently $k(q) = 1$ on an open dense set of points $q\in M$. Points where $\dim \cD_q < \dim M$ is called the singular set and we denote it as $\cZ$.
\end{definition}

From here on we will consider 2D compact almost-Riemannian manifolds without boundary satisfying the following genericity assumption (H0) from \cite[ Proposition 2]{ar2}:

\begin{enumerate}
\item $\cZ$ is an embedded one-dimensional smooth submanifold of $M$;
\item The points $q\in M$ at which $\cD^2_q$ is one-dimensional are isolated;
\item $\cD^3_q = T_q M$ for all $q\in M$.
\end{enumerate}

Locally we can take two orthonormal sections $\sigma_1,\sigma_2 \in \Gamma(\U)$ and construct two vector fields $X_i = f(\sigma_i)$ which generate our distribution in a neighbourhood of a given point $q\in M$. Note that $X_1$ and $X_2$ can not vanish at the same point as it would violate the H\"ormander condition. Thus for every point $q\in M$ there exists an open neighbourhood $O_q$ and a non-vanishing vector field in $O_q$ which we call again $X_1$. By rectifying this vector field we can find local coordinates $(x,y)$ on $O_q$ centred at $q$ such that
\begin{equation}
\label{eq:normal}
X_1 = \p_x, \qquad X_2 = f(x,y)\p_y,
\end{equation}
for some smooth function $f$. In particular $\cZ \cap O_q$ coincides with the zero locus of $f$. Under the genericity assumption $df|_{\cZ} \neq 0$. In this local frame we can write down
\begin{enumerate}
\item an expression for the metric
$$
g = \begin{pmatrix}
1 & 0 \\
0 & \frac{1}{f(x,y)^2};
\end{pmatrix}
$$
\item an expression for the associated canonical volume form:
$$
\omega = \frac{dx \wedge dy}{|f|};
$$
\item an expression for the Laplace-Beltrami operator:
\begin{equation}
\label{eq:Laplace}
\Delta = \p_x^2 + f^2 \p_y^2 -\frac{\p_x f}{f}\p_x + f(\p_y f) \p_y.
\end{equation}
\end{enumerate}
Note that all of those quantities explode at the zero locus of $f$.

We have already seen that under the genericity assumption there are three types of points that can occur. One can construct normal forms for each type which improve~\eqref{eq:normal}.
\begin{theorem}[\cite{ar1}]
\label{thm:normal_forms}
Let $q$ be a point of a generic 2D almost-Riemannian manifold. Then there exists local coordinates $(x,y)$ such that $\cD$ is locally generated by two vector fields $X_1$, $X_2$ of one of the following normal forms:
\begin{enumerate}
\item Riemannian points: 
\begin{equation}
X_1 = \p_x, \qquad X_2 = e^{\phi(x,y)}\p_y;
\end{equation}
\item Grushin points: 
\begin{equation}
\label{eq:grushin}
X_1 = \p_x, \qquad X_2 = x e^{\phi(x,y)}\p_y;
\end{equation}
\item Tangency points:
\begin{equation}
\label{eq:tangency}
X_1 = \p_x, \qquad X_2 = (y-x^2\psi(x))e^{\Psi(x,y)}\p_y;
\end{equation}
\end{enumerate}
where $\phi,\psi,\Psi$ are smooth functions, $\Psi(0,y) = 0$, $\psi(0) \neq 0$.
\end{theorem}

We will mostly work with the general form~\eqref{eq:normal} and use Theorem~\ref{thm:normal_forms} only to compute the limit operators. Next we recall all the necessary theory that would allows us to prove semi-Fredholm properties of the Laplace-Beltrami operator $\Delta$ and compute explicitly the closure of the perturbed operator.

\section{Lie groupoids, Lie algebroids and Lie manifolds}

\label{sec:lie}

\subsection{Definitions and examples}

\label{subsec:lie_def}
Lie groupoids were introduced by Charles Ehresmann and nowadays they represent an indispensable tool in many fields of mathematics like Poisson geometry~\cite{weinstein} and foliation theory~\cite{folliations}. In analysis Lie groupoids are often used in global geometry in studying various geometric objects via differential operators associated naturally to them. Probably the most famous application of this kind is given by the Connes' proof of the index theorem~\cite[Chapter 2, Section 10]{connes}. As many geometric objects are singular, Lie groupoids provide a language that allows to treat regular and singular objects in a unified manner. In the following exposition we follow closely papers~\cite{nistor_desing,nistor_fred,nistor_pdo}.

A particular type of singularity that we will be interested in are corners. A \textit{manifold $M$ with corners} is second countable, Hausdorff topological space locally modelled by open subsets of $[-1,1]^n$. We will abbreviate their names just to \textit{manifolds}. If a manifold actually does not have a boundary, we will call it a \textit{smooth manifold}. Every point $q\in M$ has an inward pointing tangent cone $T^+_q M$. In the interior of $M$ we have $T^+_q M = T_q M$. A \textit{tame submersion} $h: M_1 \to M$ is a smooth map such that its differential is surjective and $dh(v)$ is inward pointing if and only if $v \in T^+_q M_1$. Given a tame submersion each preimage $h^{-1}(q)$, $q
\in M$ is actually a smooth manifold. We will use $M_0$ to denote the interior of $M$.

We can now define Lie groupoids. Intuitively one can think of them in many ways, but probably the most useful is to look at them as manifolds with a compatible ``partial multiplication" which means that we can not multiply every element by every other element. This multiplication function must satisfy certain compatibility conditions that are listed in the following formal definition.

\begin{definition}
A \textit{Lie groupoid} $\cG \rightrightarrows M$ is a pair of manifolds $(\cG,M)$ with the following structure maps:
\begin{itemize}
\item an injection $u: M \to \cG$ called the \textit{unit map}. The space $u(M) \subset \cG$ which we identify with $M$ is called the space of units;
\item a pair of tame submersions $d,r: \cG \to M$ called the \textit{domain} (or source) and \textit{range} (or target) maps such that $d\circ u = \id$, $r\circ u = \id$;
\item a \textit{multiplication map} $m: \cG^{(2)} \to \cG$ from the set of composable pairs: 
$$
\cG^{(2)} = \{(g',g)\in \cG\times \cG \, : \, r(g)=d(g')\},
$$
and analogously to the theory of Lie groups we shorten $m(g,g')$ to $gg'$. The multiplication map satisfies:
$$
d(g'g)=d(g), \qquad r(g'g)=r(g'),
$$ 
$$
(g''g')g=g''(g'g), \qquad g d(g) = g, \qquad r(g) g = g;
$$
\item an \textit{inversion map} $i: \cG \to \cG$, $i: g\mapsto g^{-1}$ such that 
$$
gg^{-1} = r(g), \qquad g^{-1}g = d(g). 
$$
\end{itemize}
\end{definition}

\begin{remark}
Some modifications are possible. For example, $\cG$ often is not assumed to be Hausdorff even though $M$ is always assumed to be Hausdorff. In this paper we will only work with Hausdorff Lie groupoids.
\end{remark}

\begin{example}
\label{ex:lie}
Let $G$ be a Lie group. We can consider it as a Lie groupoid $G\rightrightarrows \{\id_G\}$ with multiplication given by the usual Lie group multiplication. The groupoid definition of a Lie group emphasises more its algebraic structure than its geometric structures.
\end{example}

\begin{example}
\label{ex:man}
Let $M$ be a manifold. Then we can consider a Lie groupoid $M\rightrightarrows M$, where $r(q)=d(q) =q $ for all $q\in M$ and hence all multiplications are just multiplications by units.
\end{example}

For this reason one sometimes says that Lie groupoids interpolate between manifolds and Lie groups. Let us see some other examples which lie in between.

\begin{example}
Let $G$ be a Lie group acting on a manifold $M$. It is known that the quotient $M\backslash G$ is not always a manifold. For this reason it is often more natural to consider the action groupoid $G\ltimes M \rightrightarrows M $ which topologically is just $M \times G$. The domain and range maps are given by $d(g,q) = q$, $r(g,q)=gq$ and the multiplication is $(g',g q),(g,q) = (g'g,q)$. 
\end{example}

\begin{example}
\label{ex:pair}
Another important example is the pair groupoid $M\times M \rightrightarrows M$. The domain and the range are given by $d(q',q)=q$, $r(q',q)=q'$ and the multiplication by $(q'',q'),(q',q)=(q'',q)$. Under a suitable natural notion of isomorphisms of groupoids~\cite{lg} one can show that for a Lie group $G$ the pair groupoid $G \times G$ is isomorphic to the action groupoid $G\ltimes G$. 
\end{example}

The pair gropoid $\R \times \R \rightrightarrows \R$ is a great visual aid for understanding groupoids and their multiplication. The space of units for this groupoid is given by the diagonal, range and domain maps are projections to the diagonal parallel to one of the axis. The multiplication of two elements $g'g = g''$ is depicted in Figure~\ref{fig:rgropoid}.

\begin{center}
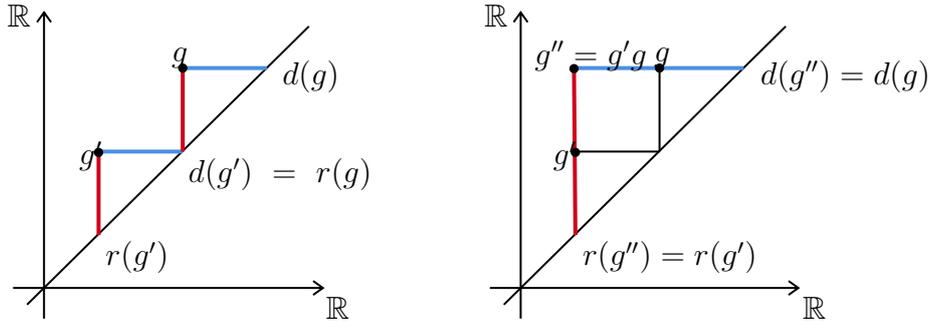
\begin{figure}[h]

\begin{center}

\tikzset{every picture/.style={line width=0.75pt}} 

\begin{tikzpicture}[x=0.75pt,y=0.75pt,yscale=-0.7,xscale=0.7]

\draw  (60,228) -- (281,228)(82.1,30) -- (82.1,250) (274,223) -- (281,228) -- (274,233) (77.1,37) -- (82.1,30) -- (87.1,37)  ;
\draw    (271,40) -- (70,240) ;
\draw [color={rgb, 255:red, 208; green, 2; blue, 27 }  ,draw opacity=1 ][line width=1.5]    (121,130) -- (121,190) ;
\draw [color={rgb, 255:red, 74; green, 144; blue, 226 }  ,draw opacity=1 ][line width=1.5]    (121,130) -- (181,130) ;
\draw [color={rgb, 255:red, 208; green, 2; blue, 27 }  ,draw opacity=1 ][line width=1.5]    (181,70) -- (181,130) ;
\draw [color={rgb, 255:red, 74; green, 144; blue, 226 }  ,draw opacity=1 ][line width=1.5]    (181,70) -- (241,70) ;
\draw  [fill={rgb, 255:red, 0; green, 0; blue, 0 }  ,fill opacity=1 ] (118.4,130.6) .. controls (118.4,129.16) and (119.56,128) .. (121,128) .. controls (122.44,128) and (123.6,129.16) .. (123.6,130.6) .. controls (123.6,132.04) and (122.44,133.2) .. (121,133.2) .. controls (119.56,133.2) and (118.4,132.04) .. (118.4,130.6) -- cycle ;
\draw  [fill={rgb, 255:red, 0; green, 0; blue, 0 }  ,fill opacity=1 ] (178.4,70) .. controls (178.4,68.56) and (179.56,67.4) .. (181,67.4) .. controls (182.44,67.4) and (183.6,68.56) .. (183.6,70) .. controls (183.6,71.44) and (182.44,72.6) .. (181,72.6) .. controls (179.56,72.6) and (178.4,71.44) .. (178.4,70) -- cycle ;
\draw  (400,228) -- (620,228)(422,30) -- (422,250) (613,223) -- (620,228) -- (613,233) (417,37) -- (422,30) -- (427,37)  ;
\draw    (611,40) -- (410,240) ;
\draw [color={rgb, 255:red, 208; green, 2; blue, 27 }  ,draw opacity=1 ][line width=1.5]    (460,70) -- (461,190) ;
\draw    (461,130) -- (521,130) ;
\draw    (521,70) -- (521,130) ;
\draw [color={rgb, 255:red, 74; green, 144; blue, 226 }  ,draw opacity=1 ][line width=1.5]    (460,70) -- (581,70) ;
\draw  [fill={rgb, 255:red, 0; green, 0; blue, 0 }  ,fill opacity=1 ] (458.4,130.6) .. controls (458.4,129.16) and (459.56,128) .. (461,128) .. controls (462.44,128) and (463.6,129.16) .. (463.6,130.6) .. controls (463.6,132.04) and (462.44,133.2) .. (461,133.2) .. controls (459.56,133.2) and (458.4,132.04) .. (458.4,130.6) -- cycle ;
\draw  [fill={rgb, 255:red, 0; green, 0; blue, 0 }  ,fill opacity=1 ] (518.4,70) .. controls (518.4,68.56) and (519.56,67.4) .. (521,67.4) .. controls (522.44,67.4) and (523.6,68.56) .. (523.6,70) .. controls (523.6,71.44) and (522.44,72.6) .. (521,72.6) .. controls (519.56,72.6) and (518.4,71.44) .. (518.4,70) -- cycle ;
\draw  [fill={rgb, 255:red, 0; green, 0; blue, 0 }  ,fill opacity=1 ] (457.4,70.4) .. controls (457.4,68.96) and (458.56,67.8) .. (460,67.8) .. controls (461.44,67.8) and (462.6,68.96) .. (462.6,70.4) .. controls (462.6,71.84) and (461.44,73) .. (460,73) .. controls (458.56,73) and (457.4,71.84) .. (457.4,70.4) -- cycle ;

\draw (54,22.4) node [anchor=north west][inner sep=0.75pt]    {$\mathbb{R}$};
\draw (281,232.4) node [anchor=north west][inner sep=0.75pt]    {$\mathbb{R}$};
\draw (105,120.4) node [anchor=north west][inner sep=0.75pt]    {$g' $};
\draw (171,53) node [anchor=north west][inner sep=0.75pt]    {$g $};
\draw (250,62.4) node [anchor=north west][inner sep=0.75pt]    {$d( g )$};
\draw (183,132.4) node [anchor=north west][inner sep=0.75pt]    {$d( g' ) \ =\ r( g )$};
\draw (124,192.4) node [anchor=north west][inner sep=0.75pt]    {$r( g' )$};
\draw (394,22.4) node [anchor=north west][inner sep=0.75pt]    {$\mathbb{R}$};
\draw (621,232.4) node [anchor=north west][inner sep=0.75pt]    {$\mathbb{R}$};
\draw (443,120.4) node [anchor=north west][inner sep=0.75pt]    {$g' $};
\draw (515,53) node [anchor=north west][inner sep=0.75pt]    {$g $};
\draw (591,62.4) node [anchor=north west][inner sep=0.75pt]    {$d( g'' ) =d( g )$};
\draw (464,192.4) node [anchor=north west][inner sep=0.75pt]    {$r( g'' ) =r( g' )$};
\draw (430,48.4) node [anchor=north west][inner sep=0.75pt]    {$g'' =g'g $};

\end{tikzpicture}

\end{center}

\caption{Range, domain and multiplication in the pair groupoid $\R\times \R \rightrightarrows \R$. \label{fig:rgropoid}}

\end{figure}
\end{center}

\vspace{-10pt}
Given two subsets $U,V\subset M$ we define the following subsets of $\cG$: $\cG_U = d^{-1}(U)$, $\cG^V = r^{-1}(V)$, $\cG^V_U = d^{-1}(U)\cap r^{-1}(V)$. We call $\cG_U^U$ the \textit{reduction} of $\cG$ to $U$. If both $U$ and $\cG_U^U$ are manifolds, then $\cG_U^U$ is called the \textit{reduced groupoid} or \textit{reduction} to $U$. In order to simplify notations we will denote the reduced groupoid $\cG^U_U$ by $\cG|_U$. As a special case to every point $q\in M$ we can associate a Lie group $G_q = \cG|_q = d^{-1}(q) \cap r^{-1}(q)$ known as the \textit{isotropy group} of $q$. If $U$ is $\cG$-invariant, meaning that $\cG^U_U = \cG^U = \cG_U$, then the reduction $\cG_U$ is a Lie groupoid called the \textit{restriction} of $\cG$ to $U$. Reduced groupoids will play an important role in the construction of Lie groupoids, while restrictions and isotropy groups will be used for defining the limit operators.

Next we pass to the definition of a Lie algebroid.

\begin{definition}
A \textit{Lie algebroid} is a triple $(A,[\cdot,\cdot],\rho)$ consisting of a vector bundle $A$ endowed with a Lie bracket $[\cdot,\cdot]$ and a morphism of vector bundles $\rho: A \to TM$ called the anchor map which satisfies the following identities:
\begin{itemize}
\item The Leibnitz rule: $[X,fY]= (\rho(X)f) Y + f[X,Y]$, where $X,Y \in \Gamma(A)$ and $f\in C^\infty(M)$.
\item Lie algebra homorphism: $\rho([X,Y])=[\rho(X),\rho(Y)]$.
\end{itemize}
\end{definition}

Similarly to Lie groups we can associate the right multiplication map $R_g: \cG_{r(g)} \to \cG_{d(g)}$ as
$$
R_g : h \mapsto hg.
$$
and it establishes a diffeormophism between $\cG_{r(g)}$ and $\cG_{d(g)}$. This allows us to push vectors between the corresponding tangent bundles and define the notion of right invariant vector fields as sections of $\bigcup_{q\in M}T\cG_q$ which are invariant under right multiplication. Define the vector bundle $A(\cG)=\bigcup_{q\in M}T_q\cG_q = \bigcup_{q\in M}T_q \left(d^{-1}(q)\right)$. Note that there is one-to-one correspondence between right invariant vector fields and section of $A(\cG)$ exactly in the same manner as there is a one-to-one correspondence between right invariant vector fields on a Lie group and vectors from a Lie algebra. Space $A(\cG)$ is called the \textit{Lie algebroid of the Lie groupoid} $\cG\rightrightarrows M$. The Lie bracket on $A(\cG)$ is the restriction of the Lie bracket between right invariant vector fields to the space of units. The anchor map is given by $\rho = r_*|_{A(\cG)}$. The image of $\Gamma(A(\cG))$ under the anchor map we denote by $Lie(\cG)$. Finally note that equivalently $\cA(\cG)$ is the restriction to the space of units of the vector bundle $\ker d_*$, where $d_* : T\cG \to TM$ is the differential of $d$. 

\begin{example}
Let $G\rightrightarrows \{\id_G\}$ be a Lie group viewed as a Lie groupoid. Then $d^{-1}(\id_G) = G$, $A(\cG)= T_{\id_G}G =\fg$. The Lie bracket is the usual Lie algebra bracket. The anchor map then maps $T_{\id_G}G$ to $0$ as it should, since $T(\id_G)=\{0\}$. Hence $Lie(\cG)=\{0\}$.  
\end{example}

\begin{example}
Let $M\rightrightarrows M$ be a manifold viewed as a Lie groupoid. In this case $d^{-1}(q) = q$. Hence $A(\cG) = M$. The anchor map $\rho: M \to TM$ is the embedding of $M$ as the zero section.
\end{example}

\begin{example}
Let $G \ltimes M \rightrightarrows M $ be an action groupoid. We have that $d^{-1}(q) \simeq G$. And the set of right invariant vector fields coincides with vector fields $\tilde{X}(g,q)= X(g) \oplus 0$, where $X(g)$ is a right invariant vector field on $G$. Hence Lie bracket on $A(\cG)$ coincides with the point-wise Lie bracket. Elements of the Lie algebra $\fg$ can be identified with the constant sections of $A(\cG)$. Under the anchor map they are mapped to infinitesimal actions of the Lie algebra $\fg$ on $M$. The Lie bracket for general sections of $A(\cG)$ can be found in~\cite[Example 3.5.14]{mackenzie}.

\end{example}

\begin{example}
Consider the pair groupoid $M\times M \rightrightarrows M$. We have $d^{-1}(q) = M\times \{q\}$. Hence $A(\cG) = \cup_{q\in M} \left(T_q M \times  \{q\}\right) = TM$. One can identify the section of $A(\cG)$, which are vector fields $X\in\Gamma(TM)$, with right invariant vector fields $\tilde X(q',q)=  0\oplus X(q)$ and inherit a Lie bracket that is just given by the usual Lie bracket of vector fields on $TM$. Hence the anchor map is just the identity.
\end{example}

Once again the pair groupoid $\R \times \R \rightrightarrows \R$ provides a great visual add for understanding $A(\cG)$ and right invariant vector fields. In Figure~\ref{fig:right} the process of translating a vector from $A(\cG)$ inside $T\cG$ is considered. In this particular case it is simply given by parallel transport of the vector along the range fibers. Thus all right invariant vector fields are vector fields constant on the vertical line. 

\begin{center}
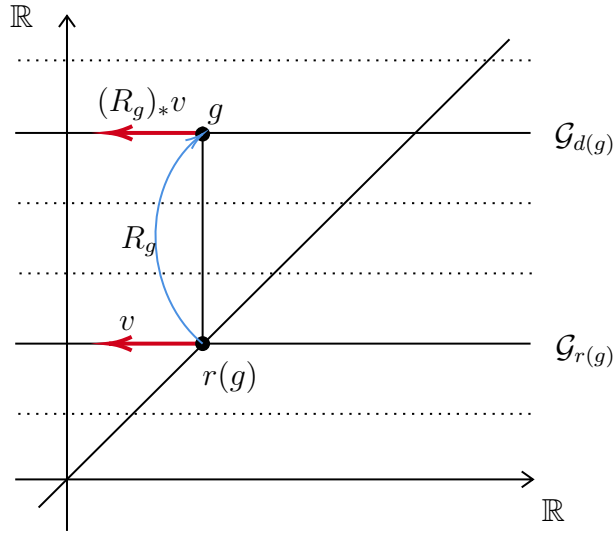
\begin{figure}

\begin{center}
\tikzset{every picture/.style={line width=0.75pt}} 

\begin{tikzpicture}[x=0.75pt,y=0.75pt,yscale=-0.8,xscale=0.8]

\draw  (103,347.62) -- (426.03,347.62)(135.3,56.19) -- (135.3,380) (419.03,342.62) -- (426.03,347.62) -- (419.03,352.62) (130.3,63.19) -- (135.3,56.19) -- (140.3,63.19)  ;
\draw    (411.41,70.91) -- (117.61,365.28) ;
\draw  [dash pattern={on 0.84pt off 2.51pt}]  (424.56,306.41) -- (103,306.41) ;
\draw    (424.56,262.25) -- (103,262.25) ;
\draw  [dash pattern={on 0.84pt off 2.51pt}]  (425.3,218.1) -- (103.73,218.1) ;
\draw  [dash pattern={on 0.84pt off 2.51pt}]  (424.56,173.94) -- (103,173.94) ;
\draw    (424.56,129.78) -- (103,129.78) ;
\draw    (219.93,125.74) -- (219.93,131.49) -- (219.93,262.25) ;
\draw  [dash pattern={on 0.84pt off 2.51pt}]  (424.56,84.16) -- (103,84.16) ;
\draw [color={rgb, 255:red, 208; green, 2; blue, 27 }  ,draw opacity=1 ][line width=1.5]    (219.93,262.25) -- (164.46,262.25) ;
\draw [shift={(161.46,262.25)}, rotate = 360] [color={rgb, 255:red, 208; green, 2; blue, 27 }  ,draw opacity=1 ][line width=1.5]    (14.21,-4.28) .. controls (9.04,-1.82) and (4.3,-0.39) .. (0,0) .. controls (4.3,0.39) and (9.04,1.82) .. (14.21,4.28)   ;
\draw [color={rgb, 255:red, 208; green, 2; blue, 27 }  ,draw opacity=1 ][line width=1.5]    (220.3,129.78) -- (164.83,129.78) ;
\draw [shift={(161.83,129.78)}, rotate = 360] [color={rgb, 255:red, 208; green, 2; blue, 27 }  ,draw opacity=1 ][line width=1.5]    (14.21,-4.28) .. controls (9.04,-1.82) and (4.3,-0.39) .. (0,0) .. controls (4.3,0.39) and (9.04,1.82) .. (14.21,4.28)   ;
\draw  [fill={rgb, 255:red, 0; green, 0; blue, 0 }  ,fill opacity=1 ] (215.91,130.49) .. controls (215.91,128.25) and (217.71,126.44) .. (219.93,126.44) .. controls (222.15,126.44) and (223.95,128.25) .. (223.95,130.49) .. controls (223.95,132.72) and (222.15,134.54) .. (219.93,134.54) .. controls (217.71,134.54) and (215.91,132.72) .. (215.91,130.49) -- cycle ;
\draw  [fill={rgb, 255:red, 0; green, 0; blue, 0 }  ,fill opacity=1 ] (215.91,262.25) .. controls (215.91,260.02) and (217.71,258.2) .. (219.93,258.2) .. controls (222.15,258.2) and (223.95,260.02) .. (223.95,262.25) .. controls (223.95,264.49) and (222.15,266.3) .. (219.93,266.3) .. controls (217.71,266.3) and (215.91,264.49) .. (215.91,262.25) -- cycle ;
\draw [color={rgb, 255:red, 74; green, 144; blue, 226 }  ,draw opacity=1 ]   (219.93,262.25) .. controls (181.92,230.48) and (179.97,161.88) .. (218.74,131.4) ;
\draw [shift={(219.93,130.49)}, rotate = 503.13] [color={rgb, 255:red, 74; green, 144; blue, 226 }  ,draw opacity=1 ][line width=0.75]    (10.93,-3.29) .. controls (6.95,-1.4) and (3.31,-0.3) .. (0,0) .. controls (3.31,0.3) and (6.95,1.4) .. (10.93,3.29)   ;

\draw (97.92,47.35) node [anchor=north west][inner sep=0.75pt]    {$\mathbb{R}$};
\draw (429.72,357.92) node [anchor=north west][inner sep=0.75pt]    {$\mathbb{R}$};
\draw (221,110) node [anchor=north west][inner sep=0.75pt]    {$g$};
\draw (437.64,122) node [anchor=north west][inner sep=0.75pt]    {$\mathcal{G}_{d( g)}$};
\draw (438.18,255) node [anchor=north west][inner sep=0.75pt]    {$\mathcal{G}_{r( g)}$};
\draw (218,271.08) node [anchor=north west][inner sep=0.75pt]    {$r( g)$};
\draw (165.93,243) node [anchor=north west][inner sep=0.75pt]    {$v$};
\draw (152,99.4) node [anchor=north west][inner sep=0.75pt]    {$( R_{g})_{*} v$};
\draw (167,186.4) node [anchor=north west][inner sep=0.75pt]    {$R_{g}$};

\end{tikzpicture}

\end{center}

\caption{Translation of vectors via right multiplication in $\R \times \R \rightrightarrows \R$ \label{fig:right}}
\end{figure}

\end{center} 

Similarly we can consider the dual Lie algebroid $p: A^*(\cG) \to M$, which is just the dual bundle to $A(\cG)\to M$. It has a canonical Poisson bracket defined as follows. Let $X\in \Gamma(A(\cG))$ and $\lambda \in A^*(\cG)$. Define a linear Hamiltonian as 
$$
h_X = \langle \lambda, X \rangle.
$$
Then the Possion bracket satisfies the following conditions
$$
\{h_X,h_Y\} = -h_{[X,Y]}, \qquad \{f \circ \pi, h_X\} = \rho(X)(f) \circ \pi
$$
for all $X,Y \in \Gamma(A(\cG))$ and $f\in C^\infty(M)$.
(Note that the minus sign comes from the right invariance).

We will deal with a special case of Lie algebroids which are known as Lie manifolds which are often used in the study of open and singular manifolds~\cite{nistor_lie}.

\begin{definition}
\label{def:lie}
A Lie manifold is a pair $(M,\cV)$ consisting of a compact manifold $M$ and $\cV \subset \Gamma(TM)$ of vector fields tangent to $\p M$ satisfying the following properties:
\begin{itemize}
\item $\cV$ is closed under the Lie bracket $[\cdot,\cdot]$ on $TM$;
\item $\cV$ is a $C^\infty(M)$-module that is generated in a neighbourhood of $q\in M_0$ by a set of linearly independent vector fields $X_1,\dots,X_{n}$;
\item in an open neighbourhood $U\subset M_0$ of $q\in M_0$ inside the interior $M_0 \subset M$ we have isomorphism of $\cV|_{U}$ and $TU$ as Lie algebras.
\end{itemize}
\end{definition}
Due to Serre-Swan theorem~\cite[Theorem 6.18]{karoubi} there exists a Lie algebroid $A_\cV$ such that $\cV = Lie(A_\cV)$ and $\rho: A_\cV \to TM$ is an isomorphism over $M_0$. It is well known that not any Lie algebroid comes from a Lie groupoid and a complete set of obstructions were found by Crainic and Fernandes~\cite{crainic}. Lie manifolds and the associated Lie algebroids can be always integrated by results of~\cite{nistor_integration,debord}. The only problem is that this groupoid may fail to be Hausdorff. Similar to the theory of Lie groups, if a Lie algebroid is integrable it has a unique $d$-simply connected Lie groupoid integrating it meaning that $\cG_q$ is simply connected for each $q\in M$. Such integrations are often called maximal and any other integration would be a quotient by a discrete, totally disconnected normal Lie subgroupoid (see, for example, \cite[Theorem 1.20]{log-symp}).

Let us see a couple of examples that do integrate to Hausdorff Lie groupoids, namely to action groupoids. These two examples will be relevant to the 1D model studied in Section~\ref{sec:1D} and in the construction of a Lie groupoid associated to a 2D AR manifold in Subsection~\ref{subsec:ar_lie}.

\begin{example}
\label{example:action_r}
Consider the half-line $\R_+$ given by $\{x\in\R : x \geq 0\}$ and the Lie manifold  structure $\cV$ given by a $C^\infty$-module generated by $x\p_x$. The Lie algebroid structure is particularly simple since it is given by the trivial bundle $A_\cV = \R \times \R_+$ and the anchor map $\rho$ maps some section $\sigma\in \Gamma(A_\cV)$ to $x\p_x$. This Lie algebroid can be integrated to an action groupoid, where the action of $\R$ on $\R_+$ is given by
$$
x\mapsto e^t x, \qquad x\in \R_+, \qquad t\in \R.
$$
Thus the Lie groupoid which integrates $\cA_\cV$ is given by
$$
(\R \times \R_+) \rightrightarrows \R_+,
$$
with domain, target and multiplication maps
$$
d(t,x) = x, \qquad r(t,x) = e^t x, \qquad (t_2,e^{t_1}x)(t_1,x)=(t_2+t_1,x).
$$
Figure~\ref{fig:action} gives a graphical interpretation of the leaves $\cG_x = d^{-1}(x)$, $\cG^x = r^{-1}(x)$.

\begin{figure}[ht]
\begin{center}

\tikzset{every picture/.style={line width=0.75pt}} 

\begin{tikzpicture}[x=0.75pt,y=0.75pt,yscale=-0.8,xscale=0.8]

\draw    (71.21,285.15) -- (71.21,18.94) ;
\draw [shift={(71.21,15.94)}, rotate = 450] [fill={rgb, 255:red, 0; green, 0; blue, 0 }  ][line width=0.08]  [draw opacity=0] (10.72,-5.15) -- (0,0) -- (10.72,5.15) -- (7.12,0) -- cycle    ;
\draw    (71.21,156.96) -- (377.85,156.96) ;
\draw [shift={(380.85,156.96)}, rotate = 180] [fill={rgb, 255:red, 0; green, 0; blue, 0 }  ][line width=0.08]  [draw opacity=0] (10.72,-5.15) -- (0,0) -- (10.72,5.15) -- (7.12,0) -- cycle    ;
\draw [color={rgb, 255:red, 213; green, 97; blue, 111 }  ,draw opacity=1 ] [dash pattern={on 0.84pt off 2.51pt}]  (87.99,28.76) .. controls (100.89,126.03) and (266.89,183.56) .. (363.22,195.41) ;
\draw [color={rgb, 255:red, 208; green, 2; blue, 27 }  ,draw opacity=1 ][line width=1.5]    (81.54,29.19) .. controls (82.4,127.58) and (252.7,211.55) .. (343.87,235.9) ;
\draw [color={rgb, 255:red, 213; green, 97; blue, 111 }  ,draw opacity=1 ] [dash pattern={on 0.84pt off 2.51pt}]  (80.25,28.76) .. controls (81.11,127.15) and (209.26,245.09) .. (303.01,270.73) ;
\draw [color={rgb, 255:red, 213; green, 97; blue, 111 }  ,draw opacity=1 ] [dash pattern={on 0.84pt off 2.51pt}]  (77.23,28.12) .. controls (84.98,130.2) and (127.12,226.34) .. (190.77,274.57) ;
\draw [color={rgb, 255:red, 213; green, 97; blue, 111 }  ,draw opacity=1 ] [dash pattern={on 0.84pt off 2.51pt}]  (74.22,28.12) .. controls (80.25,130.36) and (84.12,216.09) .. (121.53,271.85) ;
\draw [color={rgb, 255:red, 213; green, 97; blue, 111 }  ,draw opacity=1 ] [dash pattern={on 0.84pt off 2.51pt}]  (91.57,29.4) .. controls (125.55,122.5) and (258.43,143.02) .. (354.76,154.87) ;
\draw [color={rgb, 255:red, 213; green, 97; blue, 111 }  ,draw opacity=1 ] [dash pattern={on 0.84pt off 2.51pt}]  (96.73,29.72) .. controls (170.7,103.6) and (267.89,108.72) .. (366.81,116.74) ;
\draw [color={rgb, 255:red, 104; green, 159; blue, 220 }  ,draw opacity=1 ] [dash pattern={on 4.5pt off 4.5pt}]  (105.62,28.76) -- (105.62,285.15) ;
\draw [color={rgb, 255:red, 104; green, 159; blue, 220 }  ,draw opacity=1 ] [dash pattern={on 4.5pt off 4.5pt}]  (140.02,28.76) -- (140.02,285.15) ;
\draw [color={rgb, 255:red, 74; green, 144; blue, 226 }  ,draw opacity=1 ][line width=1.5]    (174.43,28.76) -- (174.43,285.15) ;
\draw [color={rgb, 255:red, 104; green, 159; blue, 220 }  ,draw opacity=1 ] [dash pattern={on 4.5pt off 4.5pt}]  (208.83,28.76) -- (208.83,285.15) ;
\draw [color={rgb, 255:red, 104; green, 159; blue, 220 }  ,draw opacity=1 ] [dash pattern={on 4.5pt off 4.5pt}]  (243.24,28.76) -- (243.24,285.15) ;
\draw [color={rgb, 255:red, 104; green, 159; blue, 220 }  ,draw opacity=1 ] [dash pattern={on 4.5pt off 4.5pt}]  (277.64,28.76) -- (277.64,285.15) ;
\draw [color={rgb, 255:red, 104; green, 159; blue, 220 }  ,draw opacity=1 ] [dash pattern={on 4.5pt off 4.5pt}]  (312.04,28.76) -- (312.04,285.15) ;
\draw [color={rgb, 255:red, 104; green, 159; blue, 220 }  ,draw opacity=1 ] [dash pattern={on 4.5pt off 4.5pt}]  (346.45,28.76) -- (346.45,285.15) ;
\draw [color={rgb, 255:red, 74; green, 144; blue, 226 }  ,draw opacity=1 ][line width=2.25]  [dash pattern={on 6.75pt off 4.5pt}]  (71.21,28.76) -- (71.21,285.15) ;
\draw [color={rgb, 255:red, 208; green, 2; blue, 27 }  ,draw opacity=1 ][line width=2.25]  [dash pattern={on 6.75pt off 4.5pt}]  (71.21,35.74) -- (71.21,283.15) ;
\draw  [fill={rgb, 255:red, 0; green, 0; blue, 0 }  ,fill opacity=1 ] (171.51,156.96) .. controls (171.51,155.35) and (172.82,154.04) .. (174.43,154.04) .. controls (176.04,154.04) and (177.34,155.35) .. (177.34,156.96) .. controls (177.34,158.57) and (176.04,159.87) .. (174.43,159.87) .. controls (172.82,159.87) and (171.51,158.57) .. (171.51,156.96) -- cycle ;

\draw (50.09,8.48) node [anchor=north west][inner sep=0.75pt]    {$t$};
\draw (382.09,155.91) node [anchor=north west][inner sep=0.75pt]    {$x$};
\draw (179.35,135.4) node [anchor=north west][inner sep=0.75pt]    {$x_{0}$};
\draw (354.4,229.95) node [anchor=north west][inner sep=0.75pt]    {$\mathcal{G}_{x_{0}}$};
\draw (179.36,10.56) node [anchor=north west][inner sep=0.75pt]    {$\mathcal{G}^{x_{0}}$};

\end{tikzpicture}

\end{center}
\caption{$d$- and $r$-fibers of the action groupoid (vertical lines and graphs of exponential functions correspondingly. The space of units is the $x$-axis)\label{fig:action}}
\end{figure}
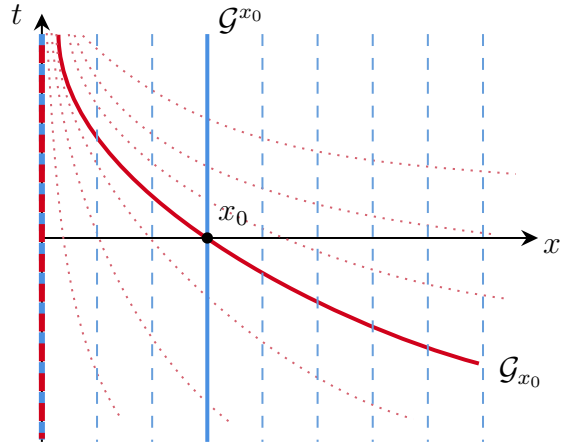

Note that $x=0$ is $\cG$-invariant subset and $\cG_0 = \cG^0$ is the isotropy group given by $\R$ with the standard addition operation. The set $x>0$ is also a $\cG$-invariant subset of $\R_+$. The restriction groupoid $\cG_{x>0}$ is equivalent to the pair groupoid. Indeed, to see this it is enough to make a change of variables
$$
(t,x) \mapsto (e^tx,x), \qquad x>0.
$$

\end{example}

\begin{example}
\label{ex:grushin}
We can consider a generalisation of the last example relevant to the study of Grushin points. Let $\R^2_+$ be the half-space given by $\{(x,y)\in \R^2\,:\, x \geq 0\}$ and the Lie manifold structure $\cV$ which is a $C^\infty$-module over $\{x\p_x,x^2\p_y\}$. Consider a solvable Lie algebra $\fg$ generated by two vector fields $X_1,X_2$ which satisfy $[X_1,X_2]=2X_2$. The Lie algebroid $A_\cV$ is a trivial bundle $\R^2_+ \times \fg$ with the anchor map given by $\rho(X_1) = x \p_x$, $\rho(X_2) = x^2\p_y$ and extended by linearity to $\Gamma(A)$. The Lie groupoid structure is given by the action groupoid $G \ltimes \R^2_+$, where $G$ is a group isomorphic to the affine group of the real line:
\begin{equation}
\label{eq:group}
G = \left\{\begin{pmatrix}
a^2 & b \\
0 & 1
\end{pmatrix} \, : \, a>0, b\in \R \right\}.
\end{equation}
The action of $G$ given by
$$
(x,y)\mapsto (ax,bx^2+y).
$$
We can identify $\R^2_+\setminus\{x=0\}$ with $G$ via a map
$$
(x,y)\mapsto \begin{pmatrix}
x^2 & y\\
0 & 1
\end{pmatrix}.
$$
On the boundary $\R$ the action is trivial, while in the interior it coincides with the right action of $G$ on itself.

As we have discussed in the Example~\ref{ex:pair} $G \ltimes G \simeq G \times G$. Therefore the $d$-simply connected Lie groupoid which integrates $A_\cV$ topologically is given by
$$
(G\times \R) \sqcup (G \times G) \rightrightarrows \R^2_+ \simeq \R^2_+.
$$
\end{example}

\subsection{Integrating Lie algebroids and glueing Lie groupoids}

\label{subsec:lie_int}

As mentioned before, in contrast to Lie algebras not any Lie algebroid can be realised as a Lie algebroid of some Lie groupoid. However this is always true for Lie algebroids coming from Lie manifolds. For our problem of determining the closure of differential operators we will require additional properties, in particular the integrating Lie groupoid should be Hausdorff.

One of the techniques of constructing new groupoids from the old ones is the \textit{gluing construction} or the \textit{fibered coproduct}, which was used successfully in~\cite{log-symp}. Let $M_1,M_2$ be two manifolds and consider two open immersions $i_j : U \hookrightarrow M_j$, $j=1,2$  of a manifold $U$. We can declare two points $q'_j \in M_j$, $j=1,2$ to be equivalent if there exist a point $q\in U$ such that $q'_j=i_j(q)$. Thus the fibred coproduct is defined as the quotient
$$
\dfrac{M_1 \sqcup M_2}{i_1(q) \sim i_2(q), \forall q\in U}
$$
and will be denoted as
$$
M_1 \sqcup M_2 /\sim
$$
Let $A \to M$ be a Lie algebroid. Assume that $M$ has an open cover $(U_i)_{i \in I}$ and that we can construct Lie groupoids $\cG_i$ which integrate the restrictions $A|_{U_i}$. A natural idea would be to glue $\cG_i$ along the reductions $(\cG_{i})|_{U_i \cap U_j}$. Unfortunately the result may fail to be a Lie groupoid. For example, consider the pair groupoid $\cG = \R \times \R$ and assume that we want to glue together its reductions $\cG|_{(-\infty,\varepsilon)}$, $\cG|_{(-\varepsilon,+\infty)}$. The result is not a Lie groupoid since the multiplication is not well-defined for all elements as can be easily seen from Figure~\ref{fig:failure}.

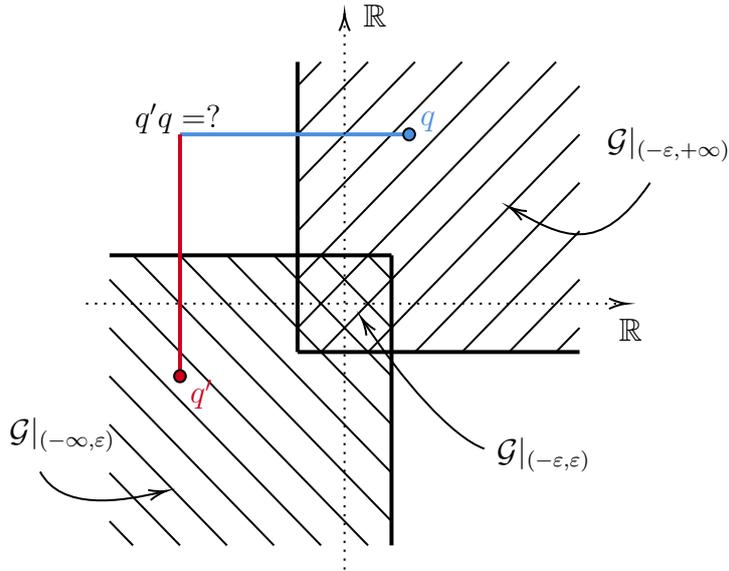
\begin{figure}[h]

\begin{center}

\tikzset{every picture/.style={line width=0.75pt}} 

\begin{tikzpicture}[x=0.75pt,y=0.75pt,yscale=-0.8,xscale=0.8]

\draw [line width=1.5]    (201.19,55.61) -- (201.19,238.03) ;
\draw [line width=1.5]    (377.18,238.03) -- (201.19,238.03) ;
\draw [line width=1.5]    (259.85,177.22) -- (259.85,359.64) ;
\draw [line width=1.5]    (259.85,177.22) -- (83.86,177.22) ;
\draw    (83.86,314.04) -- (127.86,359.64) ;
\draw    (83.86,283.63) -- (157.19,359.64) ;
\draw    (83.86,253.23) -- (186.52,359.64) ;
\draw    (83.86,222.83) -- (215.85,359.64) ;
\draw    (83.86,192.42) -- (245.19,359.64) ;
\draw    (98.52,177.22) -- (259.85,344.44) ;
\draw    (127.86,177.22) -- (259.85,314.04) ;
\draw    (157.19,177.22) -- (259.85,283.63) ;
\draw    (186.52,177.22) -- (259.85,253.23) ;
\draw    (215.85,177.22) -- (259.85,222.83) ;
\draw    (245.19,177.22) -- (259.85,192.42) ;
\draw    (377.18,70.81) -- (215.85,238.03) ;
\draw    (362.52,55.61) -- (201.19,222.83) ;
\draw    (377.18,101.21) -- (245.19,238.03) ;
\draw    (333.18,55.61) -- (201.19,192.42) ;
\draw    (377.18,131.61) -- (274.52,238.03) ;
\draw    (303.85,55.61) -- (201.19,162.02) ;
\draw    (377.18,162.02) -- (303.85,238.03) ;
\draw    (274.52,55.61) -- (201.19,131.61) ;
\draw    (377.18,192.42) -- (333.18,238.03) ;
\draw    (245.19,55.61) -- (201.19,101.21) ;
\draw    (377.18,222.83) -- (362.52,238.03) ;
\draw    (215.85,55.61) -- (201.19,70.81) ;
\draw [color={rgb, 255:red, 74; green, 144; blue, 226 }  ,draw opacity=1 ][line width=1.5]    (127.86,101.21) -- (274.52,101.21) ;
\draw [color={rgb, 255:red, 208; green, 2; blue, 27 }  ,draw opacity=1 ][line width=1.5]    (127.86,101.21) -- (127.86,253.23) ;
\draw  [dash pattern={on 0.84pt off 2.51pt}]  (69.19,207.62) -- (404.51,207.62) ;
\draw [shift={(406.51,207.62)}, rotate = 180] [color={rgb, 255:red, 0; green, 0; blue, 0 }  ][line width=0.75]    (10.93,-3.29) .. controls (6.95,-1.4) and (3.31,-0.3) .. (0,0) .. controls (3.31,0.3) and (6.95,1.4) .. (10.93,3.29)   ;
\draw  [dash pattern={on 0.84pt off 2.51pt}]  (230.52,374.84) -- (230.52,27.2) ;
\draw [shift={(230.52,25.2)}, rotate = 450] [color={rgb, 255:red, 0; green, 0; blue, 0 }  ][line width=0.75]    (10.93,-3.29) .. controls (6.95,-1.4) and (3.31,-0.3) .. (0,0) .. controls (3.31,0.3) and (6.95,1.4) .. (10.93,3.29)   ;
\draw    (40.59,313.28) .. controls (53.82,337.78) and (90,332.43) .. (117.67,325.77) ;
\draw [shift={(119.35,325.36)}, rotate = 526.21] [color={rgb, 255:red, 0; green, 0; blue, 0 }  ][line width=0.75]    (10.93,-3.29) .. controls (6.95,-1.4) and (3.31,-0.3) .. (0,0) .. controls (3.31,0.3) and (6.95,1.4) .. (10.93,3.29)   ;
\draw    (421.18,131.61) .. controls (389.4,178.43) and (366.16,165.86) .. (334.63,147.65) ;
\draw [shift={(333.18,146.82)}, rotate = 390.03] [color={rgb, 255:red, 0; green, 0; blue, 0 }  ][line width=0.75]    (10.93,-3.29) .. controls (6.95,-1.4) and (3.31,-0.3) .. (0,0) .. controls (3.31,0.3) and (6.95,1.4) .. (10.93,3.29)   ;
\draw    (317.78,298.84) .. controls (297.56,291.75) and (265.52,254.69) .. (242.14,218.33) ;
\draw [shift={(241.08,216.67)}, rotate = 417.53999999999996] [color={rgb, 255:red, 0; green, 0; blue, 0 }  ][line width=0.75]    (10.93,-3.29) .. controls (6.95,-1.4) and (3.31,-0.3) .. (0,0) .. controls (3.31,0.3) and (6.95,1.4) .. (10.93,3.29)   ;
\draw    (83.86,344.44) -- (98.52,359.64) ;
\draw  [fill={rgb, 255:red, 208; green, 2; blue, 27 }  ,fill opacity=1 ] (124.19,253.23) .. controls (124.19,251.13) and (125.83,249.43) .. (127.86,249.43) .. controls (129.88,249.43) and (131.52,251.13) .. (131.52,253.23) .. controls (131.52,255.33) and (129.88,257.03) .. (127.86,257.03) .. controls (125.83,257.03) and (124.19,255.33) .. (124.19,253.23) -- cycle ;
\draw  [fill={rgb, 255:red, 74; green, 144; blue, 226 }  ,fill opacity=1 ] (267.19,101.21) .. controls (267.19,99.11) and (268.83,97.41) .. (270.85,97.41) .. controls (272.88,97.41) and (274.52,99.11) .. (274.52,101.21) .. controls (274.52,103.31) and (272.88,105.01) .. (270.85,105.01) .. controls (268.83,105.01) and (267.19,103.31) .. (267.19,101.21) -- cycle ;

\draw (240.39,17.86) node [anchor=north west][inner sep=0.75pt]    {$\mathbb{R}$};
\draw (399.98,215.49) node [anchor=north west][inner sep=0.75pt]    {$\mathbb{R}$};
\draw (392.41,94.91) node [anchor=north west][inner sep=0.75pt]    {$\mathcal{G}|_{( -\varepsilon ,+\infty )}$};
\draw (19.49,277.33) node [anchor=north west][inner sep=0.75pt]    {$\mathcal{G}|_{( -\infty ,\varepsilon )}$};
\draw (323.65,290.14) node [anchor=north west][inner sep=0.75pt]    {$\mathcal{G}|_{( -\varepsilon ,\varepsilon )}$};
\draw (98,80) node [anchor=north west][inner sep=0.75pt]    {$q' q =?$};
\draw (276.09,85) node [anchor=north west][inner sep=0.75pt]  [color={rgb, 255:red, 74; green, 144; blue, 226 }  ,opacity=1 ]  {$q $};
\draw (132.36,253.49) node [anchor=north west][inner sep=0.75pt]  [color={rgb, 255:red, 208; green, 2; blue, 27 }  ,opacity=1 ]  {$q' $};

\end{tikzpicture}

\end{center}

\caption{Gluing reductions of a pair groupoid $\R\times \R \rightrightarrows \R$. The multiplication is only well defined inside the reduction groupoid $(-\varepsilon,\varepsilon)^2 \rightrightarrows (-\varepsilon,\varepsilon)$ \label{fig:failure}}
\end{figure}

We see now that extra conditions must be imposed on the glued Lie groupoids in order to ensure that the multiplication is well defined after taking the fibered coproduct. In~\cite{log-symp,Remi} the following natural gluing condition was proposed. 

\begin{definition}[\cite{Remi}]
\label{def:glue}
Let $(U_i)_{i\in I}$ be a locally finite open cover of $M$, $(\cG_i \rightrightarrows U_i)_\cG$ be Lie groupoids and 
$$
\varphi_{ij} : \cG_{i}|_{U_i \cap U_j} \to \cG_{j}|_{U_i \cap U_j} 
$$ 
be isomorphisms which satisfy $\varphi_{ij}=\varphi_{ji}$ and $\varphi_{ij}\varphi_{jk}=\varphi_{ik}$ for all $i,j,k \in I$ on common domains. We say that $(\cG_i)_{i\in I}$ satisfies the \textit{weak gluing condition} if for any two composable arrows $(g,g')$ of the fibered co-product
$$
\cG = \bigsqcup_{i\in I}\cG_i
$$
there exists $i\in I$ such that both $g,g'$ have a representative in $I$.
\end{definition}

For instance, consider the Lie groupoid $\cG$ from Example~\ref{example:action_r} and consider the cover $U_1= [0,\varepsilon)$, $U_2 = (0,+\infty)$. If $g,g'$ belong to $\cG|_{U_2}$, then we just use the multiplication in $\cG|_{U_2}$. If instead $g\in \cG\setminus \cG|_{U_2} = \cG_0$, then it can only be composed with $g' \in \cG_0$. But $\cG_0$ is a Lie subgroupoid of $\cG_1$ and hence the weak gluing conditions is satisfied. 

\begin{proposition}
\label{prop:glue}
Consider $M,(U_i)_{i\in I},(\cG_i)_{i\in I}$ and $\cG$ to be as in Definition~\ref{def:glue}. Then $\cG$ is a Lie groupoid. Moreover if all $(\cG_i)_{\in I}$ are Hausdorff, then $\cG$ is Hausdorff as well.
\end{proposition}

Consider again the Example~\ref{example:action_r}. Suppose that we are only given the Lie algebroid $A(\cG)$ and we do not know yet whether it is integrable or not. 
There are two $A(\cG)$-invariant subsets of $[0,+\infty)$ given by $\{0\}$ and $(0,+\infty)$. We could separately integrate the restrictions of $A(\cG)$ to the two invariant subsets and take
$$
\cG = \cG_0 \sqcup \cG_{(0,+\infty)},
$$
For now we can only deduce that $\cG$ is a groupoid, because there is yet no topology on $\cG_0 \sqcup \cG_{(0,+\infty)}$. In~\cite{nistor_integration} the author describes a way of topologising this groupoid and many others.

\begin{definition}[\cite{Remi}]
By a \textit{stratified manifold} we mean a manifold together with a disjoint decomposition
$$
M = \bigsqcup_{i\in I} S_i,
$$
by a local family of smooth manifolds such that the closure is a submanifold and each $S$ is contained in a unique open face of $M$. A stratification $\sqcup_{i\in I}S_i = M$ is called \textit{$A$-invariant} if sections $\Gamma(\rho(A))$ preserve the stratification.
\end{definition} 

The following theorem is the main result of~\cite{nistor_integration}.
\begin{theorem}[\cite{nistor_integration}]
\label{thm:glue}
Let $A\to M$ be a Lie-algebroid and $\sqcup_{i\in I} S_i$ an $A$-invariant stratification. If there exist $d$-simply connected Lie groupoids $\cG_{S_i}$ integrating $A|_{S_i}$, then 
$$
\cG = \bigsqcup_{i\in I}\cG_{S_i}
$$ 
is a Lie groupoid.
\end{theorem}

The idea of the proof goes as follows. The multiplication, domain and range maps are well defined on $\cG$. So in order to transform $\cG$ into a Lie groupoid we only need to define topology in a neighbourhood of the units and then use the multiplication map in order to extend the topology to the rest of $\cG$. 

More precisely, let $X_1,\dots,X_m$ be complete right invariant vector fields, and $Y_1,\dots, Y_n$ be complete right invariant vector fields which form a basis of $A_q$ for a fixed $q\in M$, such that $\rho(X_i)$ and $\rho(Y_j)$ are smooth sections of $TM$. If $U\subset M$ is a neighbourhood of $q$ and $B_\varepsilon$ a sufficiently small neighbourhood of $0$ in $\R^n$, we can define maps $\varphi_X^Y:B_\varepsilon \times U \to \cG$ 
\begin{equation}
\label{eq:chart}
\varphi_X^Y : (t_n,\dots,t_1,q)\mapsto e^{X_m}\circ \dots \circ e^{X_1} \circ e^{t_n Y_n} \circ \dots \circ e^{t_1 Y_1}q,
\end{equation}   
where $e^{tY}$ is the flow of a vector field $Y$ at a moment of time $t$. It should be noted that if we choose $X_1,\dots,X_m$ to be zero we essentially recover the local Lie groupoid structure constructed in~\cite{debord}. Denote by $\Phi$ all collections of maps $\varphi_X^Y$ over different choices of $X,Y$ and open neighbourhoods $U$. Then the following theorem holds.

\begin{theorem}[\cite{nistor_integration}]
\label{thm:charts}
Let $A \to M$ be a Lie algebroid and $M = \sqcup_{i\in I}S_i$ be an $A$-invariant stratification of $M$. Assume that each $A|_{S_i}$ is integrable and let $\cG_{S_i}$ be the corresponding $d$-simply connected Lie groupoids. Then
$$
\cG = \bigsqcup_{i\in I}\cG_{S_i}
$$
has a differentiable structure if, and only if, the family $\Phi$ consisting of maps~\eqref{eq:chart} is a differentiable atlas.
\end{theorem}

We state this theorem in order to emphasize that the maps~\eqref{eq:chart} are the charts of the Lie groupoid $\cG$ and hence can be used to prove Hausdorff properties of $\cG$.

\section{Analysis on Lie manifolds and Lie groupoids}

\label{sec:anal}

\subsection{Compatible structures}
\label{subsec:comp_str}

Theorem~\ref{thm:glue} and Serre-Swan theorem allow us to obtain a groupoid structure starting from a Lie manifold $(M,\cV)$. The biggest strata in this case is the interior $M_0$ and restriction $A_\cV|_{M_0}$ will be integrated to the pair groupoid $\cG_{M_0} \simeq M_0 \times M_0$. This construction and right invariance allows us to pull structures from the Lie manifold $(M,\cV)$ to the corresponding Lie groupoid.

Consider for example a metric $g_0$ defined on $M_0$. We can consider the pull-back of this metric to the Lie algebroid $\rho^* g_0$ via the anchor map. This allows to define the scalar product on the sublagebra of right invariant vector fields or one can consider it as a family of metrics on $\cG_q$ for $q\in M$. The problem is that this metric may not be well defined on all of $\cG$ because the anchor map fails to be a bijection on the boundary. In order to remedy this we need to consider a class of metric that we call \textit{compatible}, which are restrictions of a smooth metric $g$ on $A(\cG)$ to $M_0$ under the anchor map. In the language of Lie manifolds we can define them as follows. 

\begin{definition}
Let $(M,\cV)$ be a Lie manifold. A \textit{compatible} metric on $M_0$ is a metric $g_0$ such that, for any $q\in M$, we can choose the basis $X_1,\dots,X_n$ from Definition~\ref{def:lie} to be orthonormal with respect to this metric on a neighbourhood $O_q \cap M_0$ of $q$.
 \end{definition} 

Such metrics have very nice properties. With a compatible metric $(M_0,g_0)$ is a complete Riemannian manifold of infinite volume with bounded covariant derivatives of curvature. Under some additional assumptions one can also guarantee that the injectivity radius is positive~\cite{nistor_lie}. Manifolds of bounded geometry have many useful analytic properties. For example, the three definitions of Sobolev spaces (via connections, via orthonormal frames and via the Laplace-Beltrami operator) coincide~\cite{sobolev1,sobolev2}. 

In a very similar way we can consider constructions of other natural geometric objects on $\cG$. For example, since we have a metric on $A(\cG)$, we can construct a volume form as an element of $\Lambda^n A^*(\cG)$. Extending by right invariance gives a system of volume forms on $\cG_q$ which is an example of right Haar system.

\begin{definition}
A \textit{right Haar system of measures} for a locally compact groupoid $\cG$ is a family $(\lambda_q)_{q\in M}$, where $\lambda_q$ are Borel regular measures on $\cG$ with support on $\cG_q$ for every $q\in M$ and satisfying:
\begin{enumerate}
\item The continuity condition:
$$
M \ni q \to \lambda_q(\varphi) := \int_{\cG_{q}} \varphi(g)d\lambda_q(g)
$$
is continuous for every $\varphi\in C_c(\cG)$.
\item The invariance condition:
$$
\int_{\cG_{r(g)}} \varphi(hg)d\lambda_{r(g)}(h)= \int_{\cG_{d(g)}} \varphi(h)d\lambda_{d(g)}(h).
$$

\end{enumerate}

\end{definition}

In the very same spirit we can consider differential operators on $M$ whose lifts to $\cG$ are right invariant differential operators.

\begin{definition}
Let $(M,\cV)$ be a Lie manifold. The set of $\cV$-differential operators $\Diff_\cV(M)$ is the algebra of differential operators generated by compositions of vector fields from $\cV$ and multiplication by $C^\infty(M)$. In the frame described in the Definition~\ref{def:lie} elements of $\Diff_\cV(M)$ are locally generated by differential operators of the form
\begin{equation}
\label{eq:diff_el}
P =aX_{i_1}...X_{i_m}, \qquad a\in C^{\infty}(M), \qquad i_j\in\{1,\dots,n\}.
\end{equation}
The subset of  $\cV$-differential operators of order $m$ is denoted as $\Diff^m_\cV(M)$.
\end{definition}

To each element of $\cV$ we can associate a section $\sigma \in A(\cG)$, and to each section $\sigma$ we can associate a right invariant vector field $Z$ on $\cG$. Thus an element of the form~\eqref{eq:diff_el} will be mapped to an element
$$
\tilde{P} = (a\circ r) Z_{i_1}\dots Z_{i_m}.
$$
One can check that
$$
\tilde{P}(f\circ r)(g) = Pf(r(g)), \qquad \forall f \in C^\infty(M), \qquad \forall g \in\cG. 
$$

\begin{example}
If $G\rightrightarrows \id$ is a Lie group, then instead of a Haar system we have a single Haar measure, right invariant vector fields are the usual right invariant vectors fields on $G$. Even though there are no non-trivial operators on the base space $\id$, the Lie algebra $A(G)$ indeed contains interesting non-trivial objects. For example, right invariant differential operators restricted to $A(G) =T_{\id} G =\fg$ can be identified with the elements of the universal enveloping algebra $U(\fg)$.
\end{example}

\begin{example}
\label{ex:pair_op}
If we consider the pair groupoid $\cG = M\times M \rightrightarrows M$, where $M$ is a smooth manifold, then every smooth metric on $g$ is a compatible metric. Each $\cG_q$ is a copy of $M$ and the lift of the metric induces the same metric $g$ on each $\cG_q$. The same happens with other objects: right invariant vector fields are just copies of the same vector field on each $\cG_q$ and differential operators are lifted to the same differential operator on every $\cG_q$.
\end{example}

More generally it is possible to define a pseudo-differential operators and pseudo-differential calculus adapted to the Lie groupoid structure. A \textit{pseudo-differential operator} (PDO) $P$ on a Lie groupoid is a family of classical PDOs $P_q: C_c^\infty(\cG_q)\to C_c^\infty(\cG_q)$ pa\-ra\-met\-ri\-sed by points $q\in M$ satisfying certain conditions (see Definition 3.2 in~\cite{nistor_pdo}). First of all it must be a differentiable right invariant family. Secondly, if $P$ is a PDO, then $P_q$ have distribution kernels $k_q$ and \textit{support} of $P$ defined as
$$
\supp(P) = \overline{\bigcup_{q\in M}\supp(k_q)}
$$
must be a closed subset of $\{(g,g'),d(g) = d(g')\}$. It will be assumed that the \textit{reduced support} given by
$$
\{g'g^{-1}:(g,g')\in \supp(P)\}
$$
is a compact subset of $\cG$.

We will denote pseudo-differential operators of order $m$ as $\Psi^m(\cG)$ and as usual
$$
\Psi^\infty(\cG) = \bigcup_{m\in \Z}\Psi^m(\cG), \qquad \Psi^{-\infty} = \bigcap_{m\in \Z} \Psi^m(\cG).
$$
We can also define the principal symbol of a pseudodifferential $P\in \Psi^m(\cG)$ as an order $m$ homogeneous function $\sigma_m$ on $A^*(\cG) \setminus 0$:
$$
\sigma_m(P)(\xi) = \sigma_m(P_q)(\xi), \qquad \forall \xi \in A^*_q(\cG) = T_q^* \cG.
$$

The following result states that the symbol has the usual properties which allow to construct parametrixes.
\begin{theorem}[\cite{nistor_pdo}]
Let $\cG\rightrightarrows M$ be a Lie groupoid. Then $\Psi^\infty(\cG)$ is an algebra with the following properties:
\begin{enumerate}
\item The principal symbol map
$$
\sigma_m : \Psi^m(\cG) \to \cS^m_c(A^*_q(\cG))/\cS^{m-1}_c(A^*_q(\cG)).
$$
is surjective with kernel $\Psi^{m-1}(\cG)$.
\item If $P\in \Psi^m(\cG)$ and $Q\in \Psi^{m'}(\cG)$, then $PQ\in \Psi^{m+m'}(\cG)$ and satisfies $\sigma_{m+m'}(PQ)=\sigma_{m}(P)\sigma_{m'}(Q)$. Consequently $[P,Q]\in \Psi^{m+m'-1}(\cG)$ and
$$
\sigma_{m+m'-1}([P,Q])=\frac{1}{i}\{\sigma_m(P),\sigma_{m'}(Q)\},
$$
where the Poisson bracket is the one that comes from the Lie algebroid structure.
\end{enumerate}
\end{theorem}

A PDO $P\in \Psi^m(\cG)$ is called \textit{elliptic} if $\sigma_m(P)$ does not vanish on $A^*(\cG) \setminus M$. Similarly to the classical pseudo-differential calculus it is possible to construct a parametrix for an elliptic operator of order $m$ and to invert it modulo elements of $\Psi^{-\infty}(\cG)$ (see~\cite[Section 3]{wunsh} for a conceptual introduction). For future reference it should be also emphasised that right invariant differential operators belong to $\Psi^\bullet(\cG)$. 

We now see that in the case of Lie manifolds we can lift differential operators from $(M,\cV)$ to invariant differential operators on $\cG$. More generally we can start with any Lie groupoid $\cG$ and use its structure to study operators generated by vector fields from $Lie(\cG)$. We can view the lift $\tilde{P}$ of an operator $P$ as a family of (pseudo-)differential operators $\tilde{P}_q$ acting on $C^\infty(\cG_q)$. We thus can consider the restrictions of those operators to various subgroupoids. In particular, in the context of stratified Lie groupoids the following definition will be of great importance.

\begin{definition}
Let $M$ be a stratified manifold and $M_0$ its interior. Consider a Lie groupoid $\cG\rightrightarrows M$ and assume that $\cG_{M_0}$ is a Lie subgroupoid. Then operators $P_{q}$, where $q\in \p M$ are called \textit{limit operators}.

\end{definition} 

Limit operators play a central role in the study of differential operators on Lie manifolds. For example, as we have already mentioned in the introduction, one can formulate necessary and sufficient conditions for Fredholmness in terms of the invertibility of limit operators. As consequence one can also deduce information about the essential spectrum of an operator $P$ from the spectrum of its limit operators~\cite{nistor_fred}.

\subsection{Sobolev spaces}

\label{subsec:sobolev}
Given a compatible metric on a Lie manifold $(M,\cV)$, we denote the corresponding volume form as $\mu_\cV$ and the corresponding Laplace operator as $\Delta$. The $L^2$-norms in this subsection are taken with respect to $\mu_\cV$. Recall that $M_0$ is complete when endowed with a compatible metric and hence $\Delta$ is essentialy self-adjoint. We can give three different definitions of Sobolev spaces via three different norms.

First let $\nabla$ be the covariant derivative of the associated Levi-Civita connection and $k\in \N$. We can define the norm
$$
\|u\|^2_{\nabla,H^k_\cV(M)} = \sum_{i=1}^k \|\nabla^i u\|^2_{L^2_\cV(M)}.
$$
If $\cX\subset \cV$ is a finite set of vector fields such that $C^\infty(M)\cX = \cV$, then we can define
$$
\|u\|^2_{\cX,H^k_\cV(M)} = \sum \|X_1X_2...X_i u\|^2_{L^2_\cV(M)}
$$
where the sum is taken over $i\in\{1,\dots,k\}$ and all possible $X_1,X_2,\dots X_k \in \cX$. Finally we can use functional calculus to define the norms
$$
\|u\|^2_{\Delta,H^k_\cV(M)} = \sum \|(1+\Delta)^{\frac{k}{2}} u\|^2_{L^2_\cV(M)}
$$

\begin{definition}
Sobolev spaces $H^k_\cV(M)$ are defined to be completions of $C_c^\infty(M_0)$ in one of the norms above.
\end{definition}

\begin{theorem}[\cite{sobolev}]
All three Sobolev norms are equivalent. Moreover different Sobolev spaces constructed using different compatible metrics or different choices of $\cX \subset \cV$ are equivalent as well. 
\end{theorem}
In view of this theorem we will denote all three norms by $\|\cdot\|_{H^k_\cV(M)}$. We should note that one can define Sobolev spaces for every $s\in \R$ using interpolation.

In order to deal with the Laplace-Beltrami operator we need weighted Sobolev spaces. 

\begin{definition}
Fix a boundary hyperface $H$ of $M_0$ and a defining function $s_H$. Define
$$
s = \prod s^{a_H}_H,
$$
where $a_H\in \R$ and the product is taken over all hyperfaces of $M_0$. Then the weighted Sobolev spaces are defined as
$$
s H_\cV^k(M) = \{s u \, : \, u\in H_\cV^k(M)\}
$$
with the norm $\|s u\|_{s H_\cV^k(M)} = \|u\|_{H_\cV^k(M)}$.
\end{definition}

\begin{lemma}
\label{lemma:emb}
If $a_H>0$ for each hypersurface $H$ in the definition of $s$, then for any $\alpha>\alpha'$ and $0\leq m \leq k$ Sobolev space $s^\alpha H_\cV^k(M)$ is continuously embedded into $s^{\alpha'} H_\cV^{k-m}(M)$. 
\end{lemma}

The proof is a direct consequence of definitions and the fact that under the assumptions $s^{\alpha-\alpha'}$ is a smooth bounded function.

It should be noted that although for the applications that we have in mind this is all the information we need, Sobolev spaces on Lie manifolds were extensively studied with many of the classical results extended to this setting. In particular, Relich-Kondrachov theorems for general $s^\alpha W^{k,p}$ spaces hold~\cite[ Theorem 4.6]{sobolev} as well as Gagliardo-Nirenberg-Sobolev inequalities~\cite[Proposition 3.14]{sobolev}. One can also define the traces~\cite[Theorem 4.3]{sobolev} and use them for solving boundary value problems in this setting (see~\cite{nistor_lie} for applications to solutions of boundary value problems of elliptic operators on polyhedral domains).

Given a Lie manifold we can consider the associated stratified Lie groupoid $\cG \rightrightarrows M$. Assume that $\cG_{M_0}$ can be identified with the pair groupoid $M_0 \times M_0 \rightrightarrows M_0$. As discussed in Example~\ref{ex:pair_op}, each $\cG_q$ for $q\in M_0$ can be identified with a copy of $M_0$, and given an operator $P$ on $M_0$, its lift to the pair groupoid will consist of copies of operators $P_q = P$ on each $\cG_q$. From this point of view Sobolev spaces $H^k_\cV(M)$ can be seen as natural functional spaces on $\cG_q$ for $q\in M_0$.

By analogy we can define associated Sobolev spaces on any $\cG_q$ for a stratified Lie groupoid. We start with a compatible metric and construct a Laplace-Beltrami operator $\Delta$ on $M_0$. Then we lift it to the Lie groupoid $\cG$ and consider it as a family of operators $\tilde\Delta = \Delta_q$ of $\cG_q$, $q\in M$. 

\begin{definition}
Sobolev space $H_\cV^k(\cG_q)$, $q\in M$ are completions of $C^\infty_c(\cG_q)$ in the norm
$$
\|u\|^2_{H_\cV^k(\cG_q)} = \sum \|(1+\Delta_q)^{\frac{k}{2}} u\|^2_{L^2_\cV(M)}
$$
\end{definition}

Similarly to the $H_\cV^k(M)$ spaces defined above we have the following result~\cite[Corollary 7.10]{nistor_general}.

\begin{proposition}
Spaces $H_\cV^k(\cG_q)$ do not depend on the choice of a compatible metric on $M_0$. If $P\in \Psi^m(\cG)$, then $P_q: H_\cV^\alpha(\cG_q) \to H_\cV^{\alpha-m}(\cG_q)$ are continuous operators for any $\alpha\in \R$.
\end{proposition}

\section{Groupoid representations and CNQ conditions}

\label{sec:cnq}

\subsection{Groupoid $C^*$-algbebras and their representations}

\label{subsec:c-alg}

Our goal for this section is to show how CNQ conditions appear from the representation theory of groupoid $C^*$-algebras. A good reference on $C^*$-algebras is the book~\cite{dixmier}.

Recall that a \textit{$C^*$-algebra} $A$ is a complex algebra endowed with an involution $*:A \to A$ and a complete norm satisfying for all $a,b\in A$
\begin{itemize}
\item $(ab)^* = b^*a^*$;
\item $\|ab\| \leq \|a\|\|b\|$;
\item $\|a^*a\| = \|a\|^2$.
\end{itemize} 
Given a complex Hilbert space $\cH$, the space of bounded linear operators $\cL(\cH)$ with involution given by taking the adjoint is the main example of a $C^*$-algebra.  The celebrated Gelfand-Naimark theorem states that every $C^*$-algebra is $*$-isometric to such an algebra. This also motivates the following definition
\begin{definition}
A \textit{representation} of a $C^*$-algebra $A$ is a $*$-morphism $\pi: A \to \cL(\cH_\pi)$ to the space of linear operators on a Hilbert space $\cH_\pi$. We say that a representation $\pi$ of $A$ is \textit{contractive}, if $\|\pi(a)\|\leq \|a \|$.
\end{definition}

If a $C^*$-algebra $A$ is not unital, we can always use it to construct a unital $C^*$-algebra $A'$ as follows. We can adjoin to $A$ the identity element to obtain $A'$ and define $(\lambda,a)^* = (\bar \lambda, a^*)$. Then the norm on $A$ extends in unique way to $A'$. Thus for all invertibility results below if $A$ is not unital, one has to replace $A$ by $A'$.

We can associate two main $C^*$-algebras to a Lie groupoid called the \textit{full} and \textit{reduced} Lie groupoid $C^*$-algebras. Let $C_c(\cG)$ be the space of continuous complex-valued compactly supported functions on $\cG$. If $\cG$ is endowed with a right Haar system, then this space has a convolution product defined as
$$
(\varphi_1 * \varphi_2 )(g) = \int_{\cG_{d(g)}}\varphi_1(gh^{-1}) \varphi_2(h) d\lambda_{d(g)}(h)
$$

\begin{example}
For a Lie group $G \rightrightarrows \id$ (Example~\ref{ex:lie}) we have $\cG_{d(g)} = G$ and $\lambda_{d(g)} = \lambda$ is just the right invariant Haar measure. The convolution product is the standard convolution product on $G$.
\end{example}

\begin{example}
For the trivial groupoid structure (Example~\ref{ex:man}) the right Haar system is given by the singletons $\lambda_q = \delta_q$, i.e.,
$$
\lambda_q(\varphi) = \varphi(q).
$$
Then the convolution product is nothing but the pointwise multiplication in $C_c(M)$.
\end{example}

\begin{example}
Consider the pair groupoid $M\times M \rightrightarrows M$ with $M$ smooth (Example~\ref{ex:pair}). Since each $\cG_q$ is diffeomorphic to $M$ under the range map, we can take as a Haar system copies of the same smooth volume form $\omega$ on $M$. Then the convolution product transforms to
$$
(\varphi_1 * \varphi_2)(q',q) = \int_M \varphi_1(q',q'')\varphi_2(q'',q)d\omega(q''),
$$
which is a formula for the integral kernel of composition of two integral operators with kernels $\varphi_1$, $\varphi_2$.
\end{example}

The space $C_c(\cG)$ with convolution product and involution
$$
\varphi^*(g) = \overline{\varphi(g^{-1})}
$$
becomes an associative $*$-algebra. One can also endow it with a natural norm
$$
\|\varphi\|_I = \max \left\{\sup_{q\in M} \int_{\cG_q}|\varphi|d\lambda_q, \sup_{q\in M} \int_{\cG_q}|\varphi^*|d\lambda_q\right\}.
$$
Completion of $C_c(\cG)$ with respect to $\|\cdot\|_I$ is denoted by $L^1(\cG)$ 

\begin{definition}
The \textit{full $C^*$-algebra} associated to $\cG \rightrightarrows M$, denoted as $C^*(\cG)$ is the completion of $C_c(\cG)$ in the norm
\begin{equation}
\label{eq:full_norm}
\|\varphi\| = \sup_{\pi}\| \pi(\varphi)\|,
\end{equation}
where $\pi$ ranges over all contractive representations of $C_c(\cG)$, i.e., for which
$$
\|\pi(\varphi)\| \leq \|\varphi\|_I
$$
\end{definition}

This $C^*$-algebra is difficult to handle. For this reason the reduced $C^*$-algebra is often considered instead.

\begin{definition}
Given a Lie groupoid $\cG \rightrightarrows M$, a right Haar system $\lambda$ and a point $q\in M$, the \textit{regular representation} $\pi_q: C^\infty(\cG) \to \cL(L^2(\cG_q,\lambda_q))$ is defined by the formula
$$
(\pi_q(\varphi)\psi)(g) = \varphi * \psi(g), \qquad \forall \varphi \in C_c(\cG).
$$
\end{definition}

\begin{definition}
The \textit{reduced $C^*$-algebra} $C_r^*(\cG)$ is defined as the completion of $C_c(\cG)$ in the norm
$$
\|\varphi\|_r = \sup_{q\in M}\|\pi_q(\varphi)\|.
$$
\end{definition}
Since regular representations are contractive, there is a natural $*$-homomorphism $C^*(\cG) \to C^*_r(\cG)$. The groupoid $\cG$ is called \textit{metrically amenable} if this homomorphism is a bijection.

We can also construct other subalgebras of $C^*(\cG)$ and $C^*_r(\cG)$ by taking restrictions to sub-groupoids. If $A$ is a $\cG$-invariant locally closed subset of $M$, then we can define $C^*(\cG_A)$ and $C^*_r(\cG_A)$ together with $*$-homomorphisms $\rho_A: C^*(\cG) \to C^*(\cG_A)$ and $\rho_{A,r}: C^*_r(\cG) \to C^*_r(\cG_A)$.

Consider now an operator $P\in \Psi^\infty(\cG)$ and its distributional kernel $k_q$. The action of $P$ is then given by
\begin{equation}
\label{eq:kernel}
(P\varphi)(g) = \int_{\cG_{d(g)}}k_{d(g)}(g,g')\varphi(g')d\lambda_{d(g)}(g')
\end{equation}
Since $P$ are right invariant we have
$$
P \circ R_h = R_h \circ P,
$$
where $R_h$ is the multiplication on $h$ from right. Using the invariant property of Haar systems and~\eqref{eq:kernel} this invariance conditions in terms of kernels reads as
$$
k_{d(h)}(gh,g'h) = k_{r(h)}(g,g').
$$
For this reason it makes sense to define the \textit{reduced kernel}
$$
\kappa_P(g) = k_{d(g)}(g,d(g))
$$
which encodes all of the information about the corresponding operator. If we consider the convolution with the reduced kernel, we obtain using the invariance property
\begin{align*}
(\kappa_P * \varphi) (g) &= \int_{\cG_{d(g)}}k_{r(h)}(gh^{-1},r(h))\varphi(h)d\lambda_{d(g)}(h) = \\
&= \int_{\cG_{d(g)}}k_{d(h)}(g,h)\varphi(h)d\lambda_{d(g)}(h) = (P\varphi)(g).
\end{align*}
Similarly if $P,Q\in \Psi^\infty(\cG)$ are operators with reduced kernels $\kappa_P, \kappa_Q$ in the same fashion one sees that
$$
\kappa_{PQ}(g) = (\kappa_P * \kappa_Q)(g)
$$
Thus $\Psi^m(\cG)$, where $m=\pm\infty,0$ form an involutive $*$-algebra similar to compactly supported functions. For explaining CNQ conditions we will need to transform $\Psi^0(\cG)$ into a $C^*$-algebra. In order to do this we note that $\Psi^{-\infty}(\cG)$ by construction consists of smooth functions (like in the classical pseudo-differential calculus) with compact support (because of the compactness of the reduced support). Hence $\Psi^{-\infty}(\cG)$ is exactly the involutive $*$-algebra $C_c^\infty(\cG)$ and the full and reduced $C^*$-algebras are its completions. 

The following theorem was proven in \cite[Theorem 4.2]{nistor_pdo}.

\begin{theorem}
Let $(\pi,\cH)$ be a bounded representation of $\Psi^{-\infty}(\cG)$. Then it can be extended to a bounded representation of $\Psi^0(\cG)$.
\end{theorem}

\begin{definition}
\label{def:psi_compl}
We denote by $\overline{\Psi}(\cG)$ the completion of $\Psi^0(\cG)$ in the full norm~\eqref{eq:full_norm}, where $\pi$ ranges over all extensions of bounded representations of $\Psi^{-\infty}(\cG)$.
\end{definition}
The importance of $C^*$-algebra $\overline{\Psi}(\cG)$ comes from the fact that it contains resolvents of differential operators.

Many algebraic properties of $C^*$-algebras can be studied via their ideals. Two sided ideals are $C^*$-subalgebras. A two sided ideal $I$ of a $C^*$-algebra is said to be \textit{primitive}, if it is a kernel of an irreducible representation of $A$. We denote the set of primitive ideals by $\Prim(A)$. Given a representation $\pi$ of $A$ its support is defined as
$$
\supp \pi = \{J \in \Prim(A)\, : J\neq A, \, \ker \pi \subset J  \}.
$$

\begin{definition}
Let $\cF$ be a set of representations of a $C^*$-algebra $A$. $\cF$ is called \textit{exhaustive} if
$$
\Prim(A) = \bigcup_{\pi \in \cF} \supp \pi
$$ 
\end{definition}

\begin{definition}
We say that a groupoid $\cG$ has \textit{Exel's property} if the set of regular representations is exhaustive in $C^*_r(\cG)$. A groupoid $\cG$ is said to have the strong \textit{Exel's property} if the set of regular representations is exhaustive in $C^*(\cG)$.
\end{definition}

Exel's property is one of the main ingredients in the Carvalho-Nistor-Qiao Fredholm conditions. All the groupoids that we consider in this article have Exel's property as follows from Proposition 3.10 of~\cite{nistor_fred}.

\subsection{Semi-Fredholm conditions}

\label{subsec:cnq}
In this subsection we consider Lie groupoids $\cG\rightrightarrows M$ for which there exists an open dense set $U_0\subset M$, such that $\cG_{U_0} \simeq U_0 \times U_0$. All regular representations $\pi_q$ for $q\in U_0$ are thus isomorphic and we denote them by $\pi_0$. If the Lie groupoid $\cG$ is Hausdorff, then the regular representation $\pi_0$ is non-degenerate~\cite{skandalis}, and we have an isomorphism $C_r^*(\cG_{U_0})\simeq \pi_0(C_r^*(\cG_{U_0}))\simeq \cK$, where $\cK$ are compact operators on $L^2(U_0)$. Indeed, elements of $\pi_0(C_r^*(\cG_{U_0}))$ can be identified with completion in the $L^2$-norm of integral operators with compact continuous kernels which are themselves compact.

\begin{theorem}[\cite{nistor_fred}]
\label{thm:pre_cnq}
Let $\cG \rightrightarrows M$ be a Lie groupoid. Assume that 
\begin{enumerate}
\item There exists and open dense $\cG$-invariant set $U_0 \subset M$;
\item The canonical projection $C^*_r(\cG)\to C^*_r(\cG_{M\setminus U_0})$ induces an isomoprhism
$$
C^*_r(\cG)/C^*_r(\cG_{U_0}) \simeq C^*_r(\cG_{M\setminus U_0});
$$
\item $\cG_{M\setminus U_0}$ has Exel's property.
\end{enumerate}
Then for any unital $C^*$-algebra $A$ containing $C_r^*(\cG)$ as an essential ideal and for any $a\in A$ we have that $\pi_0(a)$ is Fredholm if, and only if, the image of $a$ in $A/C_r^*(\cG)$ is invertible and all $\pi_q(a)$, $q\in M\setminus U_0$, are invertible. 
\end{theorem}
For the proof of this Theorem see~\cite[Theorem 4.6]{nistor_fred}. We emphasize again that the proof uses the fact that $\pi_0$ is injective.

We will need a slightly weaker result since the operators we are interested in general are not Fredholm. We will see that they are however left semi-Fredholm. We have the following consequence of Theorem~\ref{thm:pre_cnq}.

\begin{corollary}
\label{cor:cnq1}
The statement of Theorem~\ref{thm:pre_cnq} remains true if we replace simultaneously ``Fredholm operators" with ``left (or right) semi-Fredholm" and ``invertible" with ``left (or right) invertible".
\end{corollary}

Indeed, this is a consequence of the following lemmas which are folklore and very likely to be written somewhere. We prove some of them for the sake of completeness.

\begin{lemma}
\label{lemma:inv}
Let $T: H_1 \to H_2$ is a bounded linear operator, where $H_1, H_2$ are Hilbert spaces. Then the following statements are equivalent:
\begin{enumerate}
\item $T$ is left invertible;
\item $\ker T = \{0\}$, $\ran T$ is closed;
\item There exists a constant $c>0$ such that
\begin{equation}
\label{eq:semibounded}
\|Tu\|_{H_2} \geq c\|u\|_{H_1}, \qquad \forall u\in H_1.
\end{equation}
\end{enumerate}  
\end{lemma}

The proof can be found in Section 4.5 of~\cite{aubin}.

\begin{lemma}
\label{lemma:inv_cond}
An element $a$ of a $C^*$-algebra $A$ is left invertible if, and only if, $a^*a$ is invertible. Similarly $a\in A$ is right invertible if, and only if, $aa^*$ is invertible.
\end{lemma}

\begin{proof}
It is obvious that if $a\in A$ is right invertible, then $a^*$ is left invertible. So it is sufficient to prove the condition for left invertibility. If $a^*a$ is invertible, then we can simply take $(a^*a)^{-1}a^*$ as the left inverse of $a$. 

To prove the other direction we use the Gelfand-Naimark theorem and assume that $a\in L(\cH)$ for some Hilbert space $\cH$ and that $a$ is left invertible. From the previous Lemma it follows that $\ker a = \{0\}$. Clearly $ \ker a \subset \ker (a^*a) $. We claim that this inclusion is actually an equality. Indeed, if $v\in \ker (a^*a)$, then for any $w\in \cH$
$$
0 = \langle w, a^*a v \rangle_\cH = \langle a w, a v \rangle_\cH  \iff av \in \ran(a)^\perp, 
$$
which is possible only if $av = 0$. Since $a^*a \in L(\cH)$ is a self-adjoint operator $\ran(a^*a)^\perp = \ker a^* a = \{0\}$. Hence $a^*a$ has zero kernel and cokernel and therefore it is invertible by the bounded inverse theorem.
\end{proof}

\begin{lemma}
\label{lemma:fred_cond}
A continuous linear operator $T: H_1 \to H_2$ between two Hilbert spaces is left (right) semi-Fredholm if, and only if, $T^* T$ (correspondingly $TT^*$) is Fredholm.
\end{lemma}

\begin{proof}
By Atkinson's theorem an operator is left (right) semi-Fredholm if it is left (right) invertible modulo compact operators. Compact operators $K(H_1,H_2)$ form a two-sided ideal in $L(H_1,H_2)$ and hence the quotient of spaces $L(H_1,H_2)/K(H_1,H_2)$ is a $C^*$-algebra known as the \textit{Calkin algebra}. Since the projection operator from $L(H_1,H_2)$ to the Calkin algebra is a $*$-homomorphism, the result is a direct consequence of Lemma~\ref{lemma:inv_cond}.
\end{proof}

Corollary~\ref{cor:cnq1} now follows directly from Lemmas~\ref{lemma:inv_cond} and~\ref{lemma:fred_cond}.

We are now ready to explain how the CNQ conditions are derived. The following Theorem is an analogue of Theorem 4.12 in~\cite{nistor_fred} and a direct consequence of Theorem~\ref{thm:pre_cnq} and Corollary~\ref{cor:cnq1}.

\begin{theorem}[Modified Carvalho-Nistor-Qiao conditions]
\label{thm:cnq}
Let $(M,\cV)$ be a compact Lie manifold and $\cG \rightrightarrows M$ be the associated Hausdorff Lie groupoid which satisfies conditions of Theorem~\ref{thm:pre_cnq} with $U_0 = M_0$. Let $\alpha\in\R$ and assume that $P_0 \in \Psi^m(M_0)$ is such that $P_0 = \pi_0(P)$ for some $P\in \Psi^m(\cG)$. Then
$$
P_0: H^\alpha_\cV(M) \to H^{\alpha-m}_\cV(M) 
$$
is left semi-Fredholm if $P_0$ is elliptic and
$$
P_q: H^\alpha(\cG_q) \to H^{\alpha-m}(\cG_q) \
$$
are left invertible for every $q\in M\setminus M_0$.
\end{theorem}

\begin{proof}
The idea of proof as one can see is to consider a PDO $P_0\in \Psi^m(M_0)$ as an operator $\pi_0(P)$, where $P\in \Psi^m(\cG)$. Then we can apply Theorem~\ref{thm:pre_cnq} to $P$. There are however some technical issues. First of all we need to reformulate the question of determining whether $\pi_0(P)$ is left semi-Fredholm as a question about elements of some $C^*$-algebra. For this reason we use a smooth metric on $A(\cG)$ and construct the corresponding right invariant Laplace operator $\Delta$. One then can replace $P$ with 
$$
a= (1+\Delta)^{(s-m)/2}P(1+\Delta)^{-s/2}.
$$

Operator $a$ now belongs to $\Psi^0(\cG)$. Note that $(1+\Delta_q)^{1/2}$ by the definitions is an isometry between $H^k(\cG_q)$ and $H^{k-1}(\cG_q)$. Hence everything we say about $\pi_q(a)$ can be transformed to statements about $\pi_q(P)$ after a suitable change of functional spaces. We view $a$ as an element of the completion $\overline{\Psi}(\cG)$ of $\Psi^0(\cG)$ from Definition~\ref{def:psi_compl}. We can now apply Theorem~\ref{thm:pre_cnq} and Corollary~\ref{cor:cnq1} with $A = \overline\Psi(\cG)$ (see Definition~\ref{def:psi_compl}). We obtain that $\pi_0(a)$ is left semi-Fredholm if, and only if, $\pi_q(a): L^2(\cG_q) \to L^2(\cG_q)$ are left invertible and the image of $a$ is left invertible in $\overline{\Psi}(\cG)/C^*_r(\cG)$. left invertibility of $\pi_q(a)$ is a part of the statement. So it only remain to prove left invertibility of $a$ in $\overline{\Psi}(\cG)/C^*_r(\cG)$. The fact that $P$ is elliptic implies that $a$ is elliptic. But then we can use the symbolic calculus for constructing a parametrix of an elliptic operator and invert it modulo $\Psi^{-\infty}(\cG) \subset C^*_r(\cG)$. Hence $a$ is invertible in $\overline{\Psi}(\cG)/C^*_r(\cG)$ and in particular left invertible.
\end{proof}

\begin{remark}
\label{rem:strat_group}
In~\cite{nistor_fred} the authors introduce a special class of groupoids for which the conditions of Theorem~\ref{thm:pre_cnq} are satisfied which they call \textit{(amenable) stratified submersion groupoids}. Those are Lie groupoids $\cG \rightrightarrows M$ for which $M$ has a $\cG$-invariant stratification
$$
\emptyset \subset U_N \subset U_{N-1} \subset \dots \subset U_0 \subset M 
$$
with $U_0$ open dense and for which moreover the restrictions of $\cG$ to each strata in $M\setminus U_0$ are certain fibered pull-back groupoids of amenable Lie group bundles. All the examples in this article are of this form since on the boundary we will always have Lie group bundles of solvable Lie groups.   
\end{remark}

\section{Outline of the method and a 1D model example}

\label{sec:1D}

Let us sum up all of the results and explain the rough informal algorithm for finding closure of singular elliptic operators using CNQ-conditions. Assume that we are given a differential operator $P$ of order $k$ on a manifold $M$ with boundary or some other singularity that we denote by $\cZ$. Let $\omega$ be a volume form on $M$ smooth outside $\cZ$. Let $M_0 = M\setminus \cZ$. In order to find the closure of $P$ defined on $C^\infty_c(M_0)$ in $L^2(M,\omega)$ topology one should follow the following steps:
\begin{enumerate}
\item Represent operator $P$ as an operator $s_{\cZ}^{-1}\Diff^k_\cV(M)$, where $(M,\cV)$ is a Lie manifold that extends $M_0$ and $s_Z$ is some function of the defining function of the singular set $\cZ$;
\item Find a compatible volume $\mu_\cV$ and represent $L^2(M_0,\omega)$ as $w L^2_\cV(M)$, where $w$ is a weight function;
\item Write down the continuous operator $\tilde{P} = s_{\cZ} w^{-1} P w:H^k_{\cV}(M) \to L^2_{\cV}(M)$;
\item Check first that $w^{-1}H^k_{\cV}(M)$ is continuously embedded in $L^2(M_0,\mu)$ as required by Proposition~\ref{prop:mendoza};
\item Check that there is a Hausdorff Lie groupoid integrating $A_\cV$ satisfying conditions of Theorem~\ref{thm:pre_cnq};
\item Write down the limit operators $\tilde P_q$ for $q\in \cZ$ and prove that they are invertible. Then from the previous point, Theorem~\ref{thm:cnq} and Proposition~\ref{prop:mendoza} it will follow that $D(\overline{P}) = w^{-1}H^k_{\cV}(M)$.
\end{enumerate}

\begin{remark}
This is just the outline of the method although in this paper we follow it word-by-word. In practice one often has to perform other actions, like doubling domains, performing blow-ups and others. 
\end{remark}

Let us illustrate the method on a relevant 1D example that will be useful for us in the study of Grushin points. Consider
\begin{equation}
\label{eq:1d_example}
\Delta = \p_x^2 -\left(\frac{3}{4}+\alpha\right) \frac{1}{x^2(1-x)^2}
\end{equation}
defined on $(0,1)$. Here $\alpha \in \R$ is an arbitrary fixed parameter. Operator $\Delta$ is a compact version of the inverse-square potential. Note that there is a symmetry $x\mapsto 1-x$, which allows to concentrate our discussion only on one of the ends.

Suppose that we are interested in the closure of this operator defined on $D(\Delta)=C^\infty_c(0,1)$ in the standard $L^2$-topology. Function $x$ is a defining function for the left boundary and function $(1-x)$ is a defining function for the right boundary. Thus we define $s=x(1-x)$. We have
$$
s^2 \Delta = x^2(1-x)^2 \p_x^2 -\left(\frac{3}{4}+\alpha\right).
$$
The Lie manifold structure $\cV$ is given by a unique vector field
$$
X = x(1-x)\p_x.
$$
Note that locally at the boundary $X$ is diffeomorphic to $x\p_x$. Hence we see that $\Delta\in s^{-2}\Diff_\cV(M)$. 

A compatible metric can be chosen to be 
$$
g = \frac{1}{x^2(1-x)^2}dx^2 
$$
and the corresponding volume form as
$$
\omega = \frac{dx}{s}.
$$
If we want the operator $\tilde P$ to have range in the standard space $L^2(0,1)$, then $s^2 \Delta$ should have range in $s^2 L^2(0,1)$ or equivalently in $s^{\frac{3}{2}} L^2_\cV(0,1)$. Now we need to get rid of the weight via conjugation by the $s^{\frac{3}{2}}$ factor. We find a slightly more general formula
\begin{align*}
\tilde P &=s^{2-\gamma}\Delta s^{\gamma} = \\
&=x^2(1-x^2) \p_x^2 + 2\gamma(1-2x) x(1-x) \p_x + \\
&+ \gamma(\gamma-1 + 2(x-1)x(2\gamma-1))-\left(\frac{3}{4}+\alpha\right) = \\
&= X^2 + (2\gamma-1)(1-2x)X+\gamma(\gamma -1+ 2(x-1)x(2\gamma-1)) -\left(\frac{3}{4}+\alpha\right),
\end{align*}
where $\gamma \in \R$.

Operator $\tilde P:H^2_\cV(0,1)\to L^2_\cV(0,1)$ is a continuous operator and hence $\Delta$ is a continuous operator from $s^{\frac{3}{2}} H^2_\cV(0,1)$ to $s^{-\frac{1}{2}}L^2_\cV(0,1) \simeq L^2(0,1) $ and the former is continuously embedded into the latter by Lemma~\ref{lemma:emb}. So by Proposition~\ref{prop:mendoza} we only need to check that $\tilde P:H^2_\cV(0,1)\to L^2_\cV(0,1)$ is left semi-Fredholm.

We will do this using CNQ conditions. The first step is to verify that there exists a Hausdorff Lie groupoid $\cG \rightrightarrows [0,1]$ integrating $A_\cV$ and such that $\cG_{(0,1)} \simeq (0,1)\times(0,1)$. There are several ways to do this. We will use Example~\ref{example:action_r} and the gluing Proposition~\ref{prop:glue}. The restriction $A_\cV|_{[0,\varepsilon)}$ coincides with the restriction of the Lie algebroid from Example~\ref{example:action_r} to the same semi-interval. Hence we can take the reduction of the action groupoid from Example~\ref{example:action_r} to the interval $[0,\varepsilon)$ as integration of $A_\cV|_{[0,\varepsilon)}$. The symmetry argument allows us to use the same action groupoid for integrating $A_\cV|_{(1-\varepsilon,1]}$. By Proposition~\ref{prop:glue} we can glue those groupoid to the pair groupoid $(0,1)^2 \rightrightarrows (0,1)$ to obtain a Hausdorff Lie groupoid $\cG$ integrating $A_\cV$. Moreover this Lie groupoid is a stratified submersion groupoid as stated in Remark~\ref{rem:strat_group}. Restrictions $\cG_0$, $\cG_1$ are just $\R$, which are amenable. Thus all the conditions of Theorem~\ref{thm:cnq} are verified and we can use the CNQ conditions in order to prove that $\tilde P$ is left semi-Fredholm. 

Since $\tilde P$ is elliptic in $(0,1)$, by Theorem~\ref{thm:cnq} it will be left semi-Fredholm if, and only if, both limit operators $\tilde P_0$ and $\tilde P_1$ are left invertible. The limit operators in the case when the restriction $\cG|_{\p M}$ is a bundle of Lie groups can be computed by replacing the generators of $\cV$ by the corresponding generators of the Lie algebra. For example, in our case for the left boundary point it means replacing $X=x\p_x$ with $Z=\p_y$, $y\in \R$ and evaluating the rest of functions at $x=0$. We only concentrate on the left boundary point, since the right boundary points is handled exactly in the same way via symmetry. We find for $\gamma =3/2$:
\begin{equation}
\label{eq:op_1d_ex}
\tilde P_0 = Z^2 + 2 Z - \alpha = \p_y^2 +2\p_y -\alpha.
\end{equation}
We only need to check that $\tilde P_0:H^2(\R)\to L^2(\R)$ is left invertible. In order to do this we use a slight generalisation of Lemma~\ref{lemma:inv} that will be also useful in the general case.

\begin{lemma}
\label{lemm:semibound}
Let $H_1$, $H_2$ be Hilbert spaces and $T:D(T) \subset H_1 \to H_2$ a closed operator. T is injective with closed range if, and only if, there exists a constant $c>0$ such that
$$
\|Tu\|_{H_2} \geq c\|u\|_{H_1}, \qquad \forall u \in D(T).
$$
\end{lemma}
This is a well known result and a proof can be found in~\cite[Proposition 2.14]{raymond}.

We apply this lemma to $\tilde P_0$. Using the fact that the Fourier transform is an isometry between $L^2$ spaces we find in the frequency domain
$$
\| \widehat{\tilde P_0 u}\|_{L^2(\R)} = \int_\R |\xi^2-2i\xi +\alpha|^2 |\hat{u}(\xi)|^2 d\xi = \int_\R \left((\xi^2 +\alpha)^2 + 4\xi^2\right) |\hat{u}(\xi)|^2 d\xi 
$$
Since expression in the brackets is a sum of two non-negative functions, it will be bounded by a constant if, and only if, the whole polynomial has no zeros. It is easy to see, that this is indeed the case if, and only if, $\alpha \neq 0$. Thus if $\alpha \neq 0$, by Lemma~\ref{lemm:semibound} we have that $\tilde P_0:H^2(\R)\to L^2(\R)$ has closed range and is injective and therefore by Lemma~\ref{lemma:inv} $\tilde P_0$ is left invertible. This proves the following proposition.

\begin{proposition}
\label{prop:1d}
Let $\Delta$ be the operator
$$
\Delta = \p_x^2 -\left(\frac{3}{4}+\alpha\right) \frac{1}{x^2(1-x)^2}
$$
defined on $C^\infty_c(0,1)$. If $\alpha \neq 0$, then $D(\overline{\Delta})$ as an operator from $L^2(0,1)$ to itself is given by $s^{\frac{3}{2}}H^2_\cV(0,1)$, where $s=x(1-x)$ and $\cV$ is the $C^\infty$-module generated by $X=s\p_x$.
\end{proposition}

Can we say something about $D(\overline{\Delta})$, when $\alpha = 0$? As we have seen, the operator is not left semi-Fredholm in this case due to the problems in the range. What one can do is to consider a smaller Hilbert space $(A,\|\cdot\|_A)$ which is continuously embedded in $L^2(0,1)$ and contains smooth functions as a dense subset. Then the closure $D_A(\overline \Delta)$ in this smaller space is contained in $D(\overline{\Delta})$ providing us with some useful information on the domain of the closure. Let us apply this idea to our example.

Suppose that $\alpha = 0$. Let $\varepsilon>0$, then $s^\varepsilon L^2(0,1)\hookrightarrow L^2(0,1)$ is a continuous map. We repeat the whole procedure one more time for this operator, but with $\gamma = \frac{3}{2}+\varepsilon$. The limit operator is then given by
$$
\tilde P_0 = \p_y^2 +2(1+\varepsilon)\p_y + \varepsilon(2+\varepsilon),
$$
which is left invertible. Thus we obtain that
$$
\bigcup_{\varepsilon>0} s^{3/2+\varepsilon}H^2_\cV(0,1)\subset D(\overline{\Delta})
$$
In a completely similar fashion we can now assume $-2<\varepsilon<0$. Then $ L^2(0,1)\subset s^\varepsilon L^2(0,1)$ and by the natural inclusions of Sobolev space with different weights we find that
$$
D(\overline{\Delta}) \subset \bigcap_{\varepsilon<0} s^{3/2+\varepsilon}H^2_\cV(0,1).
$$
Thus we have proven
\begin{proposition}
\label{prop:prop2}
In notations of Proposition~\ref{prop:1d} for $\alpha = 0$ one has
$$
\bigcup_{\varepsilon>0} s^{3/2+\varepsilon}H^2_\cV(0,1)\subset D(\overline{\Delta}) \subset \bigcap_{\varepsilon>0} s^{3/2-\varepsilon}H^2_\cV(0,1).
$$
\end{proposition}
\bigskip
Operators similar to~\eqref{eq:1d_example} are often encountered in practice and were extensively studied in the past. Let us compare Propositions~\ref{prop:1d} and~\ref{prop:prop2} with results which exist in the literature, more precisely with~\cite{georgescu3} which has a literature overview and the most up-to-date results. In~\cite{georgescu3} the authors study the operator of the form
\begin{equation}
\label{eq:georg}
L_\beta = -\p^2_x + \left( \beta - \frac{1}{4}\right)\frac{1}{x^2}
\end{equation}
even more generally with complex $\beta$. They state at the end of Section 1.2 and prove later that
\begin{enumerate}
\item if $\beta<1$ then $\overline{L}_\beta$ is Hermitian (symmetric) but not self-adjoint and its domain is given by $H^2_0(\R_+)$;
\item if $\beta = 1$ then $\overline{L}_\beta$ is self-adjoint and $H^2_0(\R_+)$ is dense in its domain;
\item if $\beta > 1$ then $\overline{L}_\beta$ is self-adjoint and its domain is given by $H^2_0(\R_+)$
\end{enumerate}
where 
$$
H^2_0(\R_+) = \{u\in H^2(\R_+) \, : \, u(0)=\p_x u(0) = 0 \}.
$$
Note that if we take $\beta = \alpha + 1$ in~\eqref{eq:georg} then we obtain an operator similar to~\eqref{eq:1d_example} and we expect that close to zero the behaviour of functions in $\overline{L}_{\alpha+1}$ and $\Delta$ should be the same.

Indeed, if $u\in H^2_0(\R_+)$ then $u=o(x^{\frac{3}{2}})$ as $x\to 0+$. On the other side close to zero $s\sim x$. Thus  
$$
\omega \sim \frac{dx}{x}, \qquad X \sim x\p_x, \qquad x\to 0+
$$
and if $u\in s^{\frac{3}{2}}H^2_\cV(0,1)$, then for $1>\varepsilon>0$ we have that
$$
\int_0^\varepsilon \left|u x^{-\frac{3}{2}} \right|^2 \frac{dx}{x} + \int_0^\varepsilon \left|x\p_x(u x^{-\frac{3}{2}}) \right|^2 \frac{dx}{x} + \int_0^\varepsilon \left|(x\p_x)^2 (u x^{-\frac{3}{2}}) \right|^2 \frac{dx}{x} < +\infty.
$$
Each of those integrals is finite if, and only if, $u(x) = o(x^\frac{3}{2})$ as $x\to 0+$ recovering the required asymptotics.

If $\alpha = 0$, then both $\Delta$ and $L_1$ are self-adjoint real symmetric operators. One can check that functions that go to zero as $O(x^{3/2})$ when $x\to 0+$ lie in the domain of the adjoint (those are $L^2$ functions mapped to $L^2$ functions) and hence lie in the domain of the closure as well by the results of~\cite{georgescu3}. This asymptotics is indeed consistent with Proposition~\ref{prop:prop2} and is also in accordance with the results of~\cite{mendoza}, which can be seen as the generalisation of operators~\eqref{eq:1d_example} and~\eqref{eq:georg}.

\section{Closure of the Laplace-Beltrami operator on generic 2D AR manifolds}

\label{sec:ar_anal}

\subsection{Almost-Riemannian manifolds as Lie manifolds and associated Lie groupoids}

\label{subsec:ar_lie}

Let us go back to the study of the Laplace operator on a generic 2D ARS structure. In this section we prove Theorem~\ref{thm:main}. We start by cutting $M$ at the singular set $\cZ$ and taking a connected component which we denote tautologically again by $M$ and whose boundary is $\p M \subset \cZ \times \cZ$ (see Figure~\ref{fig:cut}).

If we are given a local frame of orthonormal vector fields $X_1,X_2$ we can write the Laplace operator as
\begin{equation}
\label{eq:laplace_frame}
\Delta = X_1^2 + X_2^2 + \dive_\omega X_1 + \dive_\omega X_2.
\end{equation}
or if have chosen local coordinates such that $X_1,X_2$ are of the form
$$
X_1 = \p_x, \qquad X_2 = f(x,y)\p_y,
$$
the Laplace operator is given by~\eqref{eq:Laplace}. 

Let $s$ be a defining function of $\p M$. From the normal forms~\eqref{eq:grushin},~\eqref{eq:tangency} we see that for a generic structure $\pm f$ have non-zero differential and hence satisfy the definition of a defining function. Function $f$ is not global, but any defining functions locally can be written as 
$$
s = fe^{g}
$$ 
for some smooth function $g$. A simple way to define invariantly a global defining function for a generic 2D ARS is to take a smooth volume form $\nu$ and take the Radon-Nykodim derivative with respect to the Riemannian volume $\omega$. Nevertheless remember that the definition of Sobolev spaces and Fredholm properties do not depend on the particular choice of the defining function. For this reason we will assume in concrete calculations that above $g = 1$, since it does not influence the results, but greatly shortens formulas under consecutive differentiations.

The Lie manifold structure $\cV$ is the $C^\infty(M)$-module over vector fields which close to the singular set look like the span of
\begin{equation}
\label{eq:fields}
Y_1 = s X_1, \qquad Y_2 = sX_2.
\end{equation}
From formula~\eqref{eq:Laplace} it then follows that $\Delta = s^{-2}\Diff^2_\cV(M)$. 

\begin{theorem}
\label{thm:groupoid}
The Lie algebroid $A_{\cV} \to M$ can be integrated to a Lie groupoid $\cG$, such that 
\begin{enumerate}
\item $\cG$ is Hausdorff;
\item $\cG|_{M_0}$ is equivalent to the pair groupoid $M_0\times M_0\rightrightarrows M_0$;
\item If $q\in \cZ$ is a Grushin point, then $\cG_q$ is isomorphic to the isotopic to the identity component of the affine group of transformations of the real line;
\item If $q\in \cZ$ is tangency point, then $\cG_q$ is isomorphic to the abelian group $\R^2$.
\end{enumerate} 
\end{theorem}

\begin{proof}
Assume first that there are no tangency points. Since both $Y_1,Y_2$ vanish identically on $\cZ$, we have that the restriction $\cG_q$ for $q\in \cZ$ are Lie groups. If we at a Grushin point, then in a local neighbourhood by~\eqref{eq:grushin} we have $f(x,y) = xe^{\phi(x,y)}$. The $C^\infty(M)$-module locally generated by 
$$
Y_1 = x e^\phi \p_x, \qquad Y_1 = x^2 e^{2\phi} \p_y 
$$
coincides with $C^\infty(M)$-module locally generated by 
$$
\tilde{Y}_1 = x\p_x, \qquad \tilde{Y}_2 = x^2 \p_y.
$$
But those are the same generators as in the Example~\ref{ex:grushin}. This allows us immediately to construct an integrating Lie groupoid using glueing Proposition~\ref{prop:glue}. To see this consider the Lie algebroid $A_\cV$ coming from $\cV$. Cover the singular set by open sets $U_i$, $i=1,\dots, N$, such that on each $U_i$ vector fields $X_1,X_2$ are given by~\eqref{eq:grushin}. Then we can view $A_\cV|_{U_i}$ as the restriction of the Lie algebroid from the Example~\ref{ex:grushin}. But we have already integrated this Lie algebroid. So we can use reductions of the integrating Lie groupoid to $U_i$, that we call $\cG_i$ as Lie groupoids which integrate $A_\cV|_{U_i}$. Then using Proposition~\ref{prop:glue} we find that
$$
\cG = M_0 \times M_0 \sqcup \left(\bigsqcup_{i=1}^N \cG_{i}\right)/\sim
$$ 
integrates $A_\cV$, is Hausdorff and the isotropy groups $\cG_q$ for $q\in \cZ$ coincide with the affine group of transformation of the real line. 

Presence of tangency points introduces certain difficulties into the integration procedure. Theorem~\ref{thm:glue} guarantees that there is an integrating Lie groupoid, however it may fail to be Hausdorff due to the restrictive condition of being $d$-simply connected. In order to construct a Hausdorff integrating Lie groupoid also in the presence of tangency points we repeat the first step of the previous argument. We take an open cover $U_i \subset M$, $i=1,\dots,N$ of $\cZ$. Assume that $q$ is a tangency point and that it is contained in a unique $U_j$, which can be always achieved since tangecy points can not cluster.  Our goal is to construct a Hausdorff Lie groupoid which would integrate $A_\cV|_{U_j}$. 

Similarly to Grushin points we will construct a Lie groupoid whose reduction integrates $A_\cV|_{U_j}$ and whose restriction to the interior $U_j \cap M_0$ is the pair groupoid. To do this we consider a slightly bigger neighbourhood $\tilde{U}_j \supset U_j$, which satisfies the following extra assumption. The boundary of the closure of $\tilde U_j$ in $M$ has a one-dimensional face $S_1 = \tilde U_j \cap \p M$ and a one-dimensional face $S_2 \subset M_0$, which intersects transversally $\p M$ (see Figure~\ref{fig:for_proof}).

\vspace{-2cm}

\begin{figure}[ht]
\begin{center}

\tikzset{every picture/.style={line width=0.75pt}} 

\begin{tikzpicture}[x=0.75pt,y=0.75pt,yscale=-1,xscale=1]

\draw [line width=0.75]  [dash pattern={on 0.84pt off 2.51pt}]  (135.05,238.86) -- (490.5,238.86) ;
\draw [color={rgb, 255:red, 74; green, 144; blue, 226 }  ,draw opacity=0.47 ]   (214.04,238.86) .. controls (315.14,129.96) and (315.93,129.96) .. (411.51,238.86) ;
\draw [color={rgb, 255:red, 74; green, 144; blue, 226 }  ,draw opacity=0.47 ]   (253.84,238.86) .. controls (315.45,170.89) and (314.66,171.78) .. (372.32,238.86) ;
\draw [color={rgb, 255:red, 74; green, 144; blue, 226 }  ,draw opacity=0.47 ]   (293.33,238.86) .. controls (313.08,218.93) and (314.66,217.15) .. (332.83,238.86) ;
\draw [color={rgb, 255:red, 74; green, 144; blue, 226 }  ,draw opacity=0.47 ]   (174.54,238.86) .. controls (313.56,86.36) and (314.35,85.47) .. (451.01,238.86) ;
\draw [line width=0.75]  [dash pattern={on 0.84pt off 2.51pt}]  (135.05,238.86) .. controls (315.93,-48.88) and (314.35,-49.77) .. (490.5,238.86) ;
\draw [line width=1.5]    (174.54,238.86) .. controls (315.93,23.19) and (317.51,23.19) .. (451.01,238.86) ;
\draw [color={rgb, 255:red, 74; green, 144; blue, 226 }  ,draw opacity=0.47 ]   (135.05,238.86) .. controls (315.14,42.77) and (315.14,41.88) .. (490.5,238.86) ;
\draw [color={rgb, 255:red, 74; green, 144; blue, 226 }  ,draw opacity=0.47 ]   (214.89,117.24) .. controls (297.83,44.82) and (333.07,38.54) .. (423.6,130.25) ;
\draw [color={rgb, 255:red, 208; green, 2; blue, 27 }  ,draw opacity=0.48 ]   (431.26,149.89) -- (431.26,238.86) ;
\draw [color={rgb, 255:red, 208; green, 2; blue, 27 }  ,draw opacity=0.48 ]   (391.76,83.16) -- (391.76,238.86) ;
\draw [color={rgb, 255:red, 208; green, 2; blue, 27 }  ,draw opacity=0.48 ]   (352.27,38.67) -- (352.27,238.86) ;
\draw [color={rgb, 255:red, 208; green, 2; blue, 27 }  ,draw opacity=0.48 ]   (312.9,23.41) -- (312.78,238.86) ;
\draw [color={rgb, 255:red, 208; green, 2; blue, 27 }  ,draw opacity=0.48 ]   (273.28,38.67) -- (273.28,238.86) ;
\draw [color={rgb, 255:red, 208; green, 2; blue, 27 }  ,draw opacity=0.48 ]   (233.79,83.16) -- (233.79,238.86) ;
\draw [color={rgb, 255:red, 208; green, 2; blue, 27 }  ,draw opacity=0.48 ]   (194.29,149.89) -- (194.29,238.86) ;
\draw [color={rgb, 255:red, 208; green, 2; blue, 27 }  ,draw opacity=0.48 ]   (273.55,156.47) -- (273.27,148.99) ;
\draw [shift={(273.16,145.99)}, rotate = 447.86] [fill={rgb, 255:red, 208; green, 2; blue, 27 }  ,fill opacity=0.48 ][line width=0.08]  [draw opacity=0] (10.72,-5.15) -- (0,0) -- (10.72,5.15) -- (7.12,0) -- cycle    ;
\draw [color={rgb, 255:red, 208; green, 2; blue, 27 }  ,draw opacity=0.48 ]   (233.98,182.71) -- (233.76,175.68) ;
\draw [shift={(233.66,172.68)}, rotate = 448.2] [fill={rgb, 255:red, 208; green, 2; blue, 27 }  ,fill opacity=0.48 ][line width=0.08]  [draw opacity=0] (10.72,-5.15) -- (0,0) -- (10.72,5.15) -- (7.12,0) -- cycle    ;
\draw [color={rgb, 255:red, 208; green, 2; blue, 27 }  ,draw opacity=0.48 ]   (312.78,138.61) -- (312.69,131.19) ;
\draw [shift={(312.65,128.19)}, rotate = 449.33] [fill={rgb, 255:red, 208; green, 2; blue, 27 }  ,fill opacity=0.48 ][line width=0.08]  [draw opacity=0] (10.72,-5.15) -- (0,0) -- (10.72,5.15) -- (7.12,0) -- cycle    ;
\draw [color={rgb, 255:red, 208; green, 2; blue, 27 }  ,draw opacity=0.48 ]   (352.35,156.47) -- (352.2,148.99) ;
\draw [shift={(352.15,145.99)}, rotate = 448.91] [fill={rgb, 255:red, 208; green, 2; blue, 27 }  ,fill opacity=0.48 ][line width=0.08]  [draw opacity=0] (10.72,-5.15) -- (0,0) -- (10.72,5.15) -- (7.12,0) -- cycle    ;
\draw [color={rgb, 255:red, 208; green, 2; blue, 27 }  ,draw opacity=0.48 ]   (391.92,187.81) -- (391.74,182.35) ;
\draw [shift={(391.64,179.35)}, rotate = 448.13] [fill={rgb, 255:red, 208; green, 2; blue, 27 }  ,fill opacity=0.48 ][line width=0.08]  [draw opacity=0] (10.72,-5.15) -- (0,0) -- (10.72,5.15) -- (7.12,0) -- cycle    ;
\draw [line width=1.5]    (174.54,238.86) -- (451.01,238.86) ;
\draw [color={rgb, 255:red, 74; green, 144; blue, 226 }  ,draw opacity=0.52 ]   (296.11,162.66) -- (300.66,160.37) ;
\draw [shift={(303.34,159.02)}, rotate = 513.23] [fill={rgb, 255:red, 74; green, 144; blue, 226 }  ,fill opacity=0.52 ][line width=0.08]  [draw opacity=0] (10.72,-5.15) -- (0,0) -- (10.72,5.15) -- (7.12,0) -- cycle    ;
\draw [color={rgb, 255:red, 74; green, 144; blue, 226 }  ,draw opacity=0.52 ]   (296.46,195.83) -- (300.55,192.75) ;
\draw [shift={(302.95,190.94)}, rotate = 503.02] [fill={rgb, 255:red, 74; green, 144; blue, 226 }  ,fill opacity=0.52 ][line width=0.08]  [draw opacity=0] (10.72,-5.15) -- (0,0) -- (10.72,5.15) -- (7.12,0) -- cycle    ;
\draw [color={rgb, 255:red, 74; green, 144; blue, 226 }  ,draw opacity=0.52 ]   (292.99,128.77) -- (299.01,126.55) ;
\draw [shift={(301.83,125.52)}, rotate = 519.79] [fill={rgb, 255:red, 74; green, 144; blue, 226 }  ,fill opacity=0.52 ][line width=0.08]  [draw opacity=0] (10.72,-5.15) -- (0,0) -- (10.72,5.15) -- (7.12,0) -- cycle    ;
\draw [color={rgb, 255:red, 74; green, 144; blue, 226 }  ,draw opacity=0.52 ]   (290.56,96.34) -- (296.69,94.57) ;
\draw [shift={(299.57,93.74)}, rotate = 523.9300000000001] [fill={rgb, 255:red, 74; green, 144; blue, 226 }  ,fill opacity=0.52 ][line width=0.08]  [draw opacity=0] (10.72,-5.15) -- (0,0) -- (10.72,5.15) -- (7.12,0) -- cycle    ;
\draw [color={rgb, 255:red, 74; green, 144; blue, 226 }  ,draw opacity=0.52 ]   (291.95,66.09) -- (297.04,64.7) ;
\draw [shift={(299.93,63.9)}, rotate = 524.6800000000001] [fill={rgb, 255:red, 74; green, 144; blue, 226 }  ,fill opacity=0.52 ][line width=0.08]  [draw opacity=0] (10.72,-5.15) -- (0,0) -- (10.72,5.15) -- (7.12,0) -- cycle    ;
\draw    (161.59,99.78) .. controls (192.29,85.09) and (205.3,91.88) .. (219.4,101.7) ;
\draw [shift={(220.94,102.78)}, rotate = 215.19] [color={rgb, 255:red, 0; green, 0; blue, 0 }  ][line width=0.75]    (10.93,-3.29) .. controls (6.95,-1.4) and (3.31,-0.3) .. (0,0) .. controls (3.31,0.3) and (6.95,1.4) .. (10.93,3.29)   ;
\draw    (419.87,265.47) .. controls (446.14,262.5) and (451.18,254.09) .. (460.75,241.07) ;
\draw [shift={(461.81,239.63)}, rotate = 486.62] [color={rgb, 255:red, 0; green, 0; blue, 0 }  ][line width=0.75]    (10.93,-3.29) .. controls (6.95,-1.4) and (3.31,-0.3) .. (0,0) .. controls (3.31,0.3) and (6.95,1.4) .. (10.93,3.29)   ;

\draw (317.91,29.41) node [anchor=north west][inner sep=0.75pt]    {$\widetilde{\mathnormal{U}_{j}}$};
\draw (199.5,215.83) node [anchor=north west][inner sep=0.75pt]    {$U_{j}$};
\draw (140.7,90.03) node [anchor=north west][inner sep=0.75pt]    {$S_{2}$};
\draw (402.46,254.25) node [anchor=north west][inner sep=0.75pt]    {$S_{1}$};

\end{tikzpicture}

\caption{Illustration to the construction of the groupoid. Original set $U_j$ is the set bounded by thick solid lines, its extensions $\tilde U_j$ is the set bounded by the doted lines. The remaining lines are flow lines of $\tilde Y_i$.\label{fig:for_proof}}
\end{center}
\end{figure}
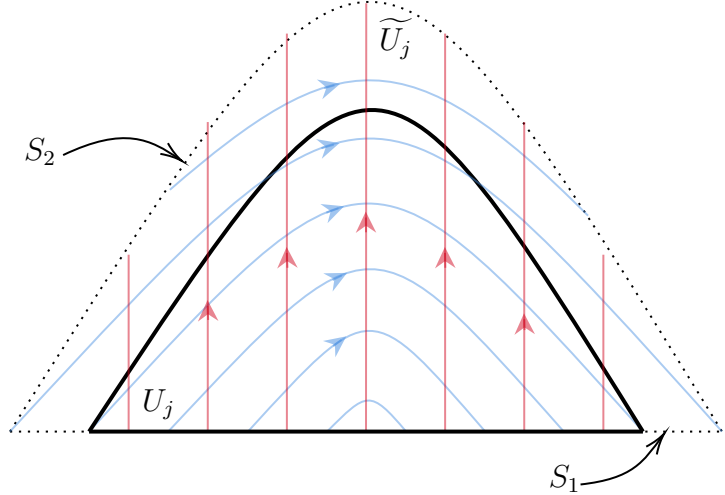

Let $s_2$ be the defining function of $S_2$, such that $s_2\equiv 1$ on $U_j$. We define a new Lie manifold structure $\tilde{\cV}$ as a $C^\infty(M)$-module generated by 
$$
\tilde Y_1 = s_2^2 Y_1, \quad \tilde Y_2 = s_2^2Y_2.
$$ 
Note that by construction $A_{\tilde{\cV}}|_{U_j} = A_\cV$. This is just a particular choice for the extension of $A_\cV$ in order to guarantee completeness of vector fields $\tilde{Y}_1$, $\tilde{Y}_2$ and other choices are, of course, are possible. We wish to apply Theorem~\ref{thm:glue} to show that there exists a Lie groupoid integrating $A_{\tilde{\cV}}$. Indeed, on $S_2$ we have
$$
[\tilde{Y}_1,\tilde{Y}_2]|_{S_2} = 0.
$$
and hence a trivial bundle of abelian groups would integrate $A_{\tilde{\cV}}|_{S_2}$. For the face $S_1$ we have
$$
[\tilde{Y}_1,\tilde{Y}_2]|_{S_1} = (X_1(s_2^2 s)\tilde{Y}_2 - X_2(s_2^2 s)\tilde{Y}_1)|_{S_1}.
$$
We can without loss of generality as discussed previously take $s=f$ and then from~\eqref{eq:normal} we find that $X_2(s_2^2 s)|_{S_1} = 0$. Let us shorten $\alpha(q) = X_1(s_2^2 s)(q)$. Then we can integrate $A_{\tilde{\cV}}|_{S_1}$ to a  trivial bundle $S_1 \times \R^2$, such that
$$
d(q,t,\tau) = r(q,t,\tau) =q , \qquad u(q)=(q,0,0)
$$
and the multiplication is given by
$$
(q,t_2,\tau_2)(q,t_2,\tau_2) = (q,t_1+t_2,e^{\alpha(q)t_2}\tau_1 + \tau_2).
$$
Hence there exists a $d$-simply connected Lie groupoid $\cG$ that integrates $A_{\tilde{\cV}}$ and such that its restriction to the interior of $\tilde{U}_j$ is the pair groupoid. It remains to show that $\cG$ is Hausdorff.

If two points belong to the restriction of $\cG$ to the interior of $\tilde U_j$ then clearly there exists two neighbourhoods separating them. The same is true if only one of the points belongs to the restriction to the interior. The only problem that may arise is that two points in $\cG_{\p \tilde{U}_j}$ maybe not separable inside $\cG$. To see that this is not the case we use Theorem~\ref{thm:charts}, which states that the maps~\eqref{eq:chart} form charts. Let $g,g'$ be points in $\cG_q$, $\cG_{q'}$ and $q,q'\in \p \tilde{U}_j$. Then we can consider two charts of the form
\begin{equation}
\label{eq:maps}
(y',t'_1,t'_2)\mapsto e^{t'_1 \tilde{Y}_1} \circ e^{t'_2 \tilde{Y}_2}(y'), \qquad (y,t_1,t_2)\mapsto  e^{t_1 \tilde{Y}_1} \circ e^{t_2 \tilde{Y}_2}(y),
\end{equation}
where $y,y'$ belong to small disjoint neighbourhoods $U,U'\subset \tilde{U}_j$ of $q,q'$. We have also slightly abused notations and identified $\tilde{Y}_i$ with right invariant vector fields using the range map. We do not need vector fields $X_i$ like in~\eqref{eq:chart}, because each of those charts already contains $\cG_{q}$ and $\cG_{q'}$ entirely. Indeed, if a right invariant vector field $Y$ is mapped to a complete vector field under the range map, then $Y$ is complete in $\cG$~\cite[Appendix, Section 33]{kumpera}. Hence by our construction maps~\eqref{eq:maps} are defined for all $t_i$, $t'_i$, $i=1,2$ and for $y=q$ and $y'=q'$ they represent coordinates of the second kind for the isotropy groups $G_q$, $G_{q'}$, which are global, since $G_{q}$ and $G_{q'}$ are solvable.

If $q=q'$, then $g,g' \in \cG_q$ are contained in a singe coordinate chart and hence can be separated by taking smaller neighbourhoods in this chart. If $q \neq q'$ we can take two disjoint neighbourhoods $U\ni q$, $U'\ni q'$ and consider charts~\eqref{eq:maps} with $y \in U$ and $y' \in U'$. We claim that we can make $U, U'$ so small that those charts do not overlap. Indeed, let us consider the images of $U\times B_\varepsilon$, $U'\times B_\varepsilon$ under~\eqref{eq:maps} projected to $\tilde{U}_j$ via the range map which we denote by $V,V'$. They will be given by the orbits of $U,U'$ under the flows of the corresponding vector fields from $\tilde{\cV}$. We assume that $\varepsilon>0$ is big enough to ensure that $g,g'$ lie in these two charts. If $U,U'$ would have been only subsets of $\p \tilde{U}_j$, they would have stayed invariant no matter how big $\epsilon$ is choosen. Thus by smooth dependence of solutions of ODEs on the initial value we can find $U,U'$ so small that $V \cap V' = \emptyset$. Hence the charts do not overlap as well and $g,g'$ are separated.

Thus we have proven that $A_{\tilde{\cV}}$ can be integrated to a Hausdorff Lie groupoid. Now it is enough to take its restriction to $U_j$ and glue to all of the other Lie groupoids obtained in a similar fashion via Proposition~\ref{prop:glue}.
 
\end{proof}

Theorem~\ref{thm:groupoid} allows us to apply the machinery of Sections~\ref{sec:anal} and~\ref{sec:cnq} in order to determine the closure of the Laplace operator $\Delta$. However $\Delta$ in many aspects is similar to the critical case $\alpha = 0$ in the 1D example~\eqref{eq:1d_example}. CNQ conditions are not directly applicable to $\Delta$, but they are applicable to certain perturbations of $\Delta$. For this reason in Theorem~\ref{thm:main} we considered instead
\begin{equation}
\label{eq:new_delta}
\tilde{\Delta} = \Delta - \frac{h}{s^2},
\end{equation}
where $h\in C^\infty(M)$ such that $h|_\cZ$ is a strictly positive function. If there are no tangency points a nice geometric perturbation of this kind exists, namely one can consider the operator $\Delta + cK$, where $c\in \R$ is a constant and $K$ is the Gaussian curvature. This operator can be considered as a possible covariant quantization of the classical energy Hamiltonian on a Grushin manifold (see~\cite{me} for further explanation).

We can now follow the algorithm outlined in Section~\ref{sec:1D}. Exactly as $\Delta$, the operator $\tilde\Delta$ belongs to $s^{-2} \Diff^2_\cV(M)$. If we want its image to lie in $L^2(M, \omega)$, then $\ran (s^2\tilde \Delta)$ must be contained in $s^2 L^2(M, \omega) = s L^2(M, \omega/s^2) $. We let $\mu_\cV = \omega/s^2$ and note that it is a volume coming from a compatible metric. In order to mitigate the weight of the $L^2$ space we conjugate by $s$ obtaining an operator 
$$
\tilde P = s\tilde \Delta s.
$$ 
Note that $\tilde P \in \Diff^2_\cV(M)$ and hence it defines a bounded operator $\tilde P: H^2_\cV(M) \to L^2_\cV(M)$. If $\tilde P$ is left semi-Fredholm, then it would imply that 
$$
\tilde{\Delta}:sH^2_\cV(M) \to s^{-1}L^2_\cV(M)\simeq L^2(M,\omega)
$$
is left semi-Fredholm as well. From Lemma~\ref{lemma:emb} it follows that there exists a continuous inclusion between $sH^2_\cV(M)$ and $L^2(M,\omega)$ and hence as consequence of Proposition~\ref{prop:mendoza} we would find that 
$$
D(\overline{\tilde{\Delta}}) = sH^2_\cV(M)
$$
thus proving Theorem~\ref{thm:main}.

So it only remains to prove that $\tilde P$ is a left semi-Fredholm. The proof will be a consequence of the modified Carvalho-Nistor-Qiao conditions stated in Theorem~\ref{thm:cnq}. Exactly as in the 1D model example we only need to check the left invertibility of limit operators $\tilde P_q$, which we study in the next subsection.

\subsection{Limit operators and their invertibility}

In a local trivialisation $X_1,X_2$, from~\eqref{eq:laplace_frame} and~\eqref{eq:fields} we find that
\begin{align*}
s\Delta s &= \sum_{i=1}^2 Y_i^2  + (X_i(s) + s \dive X_i)Y_i + (sX_i^2(s) +s X_i(s) \dive X_i). 
\end{align*}
Note that now all of the coefficients are smooth on the singular set $\cZ$ as can be easily seen from the local coordinate expressions and moreover the last term is exactly $s \Delta (s)$, which means that the free term is a globally defined smooth function.

Let us now write down the boundary operators. Let $q\in \p M$, $\cG_q$ be the restriction of $\cG$ to $q$, which coincides with the isotropy group $G_q$ and let $\fg_q$ be its Lie algebra. Suppose that $Z_1,Z_2$ is a basis of left invariant vector fields on $G_q$ at a point $q\in \p M$, which correspond to the vector fields $Y_1,Y_2$. Then a monomial $aY_{i_1}Y_{i_2}\dots Y_{i_n}$ will correspond to $a(q)Z_{i_1}Z_{i_2}\dots Z_{i_n}$, which is an element of the universal enveloping algebra $U(\fg_q)$. Since as discussed earlier the Sobolev spaces $H^k_\cV(M)$ do not depend on the choice of the compatible metric or the defining function the semi-Fredholm property of the operator $\tilde P$ does not depend as well on those things. Thus we can assume that around $q$ the defining function $s$ coincides with $f$. Recall that in coordinates around $q\in \p M$ centered at zero we have $\dive X_1 = -\p_x f/f$, $\dive X_2 = f \p_y f$ and $f(0,0) = 0$. So for the operator $\tilde P$ we obtain 
$$
\tilde P_q = Z_1^2 +Z_2^2   - \p_x f(0,0)^2 -h(0,0).
$$

Assume first that $q$ is a tangency point. From the normal form~\eqref{eq:tangency} it follows that $f(0,0) = 0$. Since $G_q = \R^2$ and
$$
[Y_1,Y_2](0,0) = 0, 
$$
we have
$$
\tilde P_q = \Delta_{\R^2} - h(0,0).
$$
Thus a Fourier transform arguments proves the following proposition.

\begin{proposition}
If $q$ is tangency point, then $G_q = \R^2$ and $\tilde P_q:H^2(\R^2)\to L^2(\R^2)$ is left invertible if, and only if, $h(0,0)> 0$.
\end{proposition}

The rest of this section is dedicated to the proof of an analogous result when $q$ is a Grushin point. More precisely

\begin{proposition}
\label{prop:grushin_suffering}
If $q$ is a Grushin point, then $G_q$ is the isotopic to the identity component of the affine group of transformations of the real line and $\tilde P_q:H^2(G_q)\to L^2(G_q)$ is left invertible if, and only if, $h(0,0)\neq 0$.
\end{proposition}

For the rest of this subsection we will write $G$ instead of $G_q$ and $\fg$ instead of $\fg_q$. Spaces $H^k(G)$ are the corresponding Sobolev spaces with respect to right invariant Haar measure $\mu_G$. Recall that we can define these Sobolev spaces as the completion of compactly supported smooth functions in the Sobolev norm
$$
\|u\|^2_{H^k(G)} = \|(1+\Delta_G)^{\frac{k}{2}}u\|^2_{L^2(G)}
$$
where $\Delta$ is the invariant Laplace operator, or
$$
\|u\|^2_{H^k(G)} = \sum \|Z_{i_1}Z_{i_2}\dots Z_{i_m}u \|^2_{L^2(G)},
$$
where $m\in\{1,\dots,k\}$ and $Z_{i_j}$, ${i_j}\in\{1,\dots \dim \fg\}$ is a basis of right invariant vector fields on $G$. Since Lie groups with left invariant metrics are geodesically complete and have constant curvature, they are manifolds of bounded geometry and the two norms are equivalent. 

From the normal form~\eqref{eq:grushin} it follows that in a chart centred at a Grushin point
$$
[Y_1,Y_2](0,0) = 2Y_2(0,0).
$$
Thus
$$
\tilde P_q = Z_1^2 +Z_2^2 - 1 - h_0
$$
where $Z_1,Z_2$ are two right invariant vector fields on $G$ which satisfy
$$
[Z_1,Z_2]= 2Z_2.
$$
It turns out that the right invariant definitions in the particular instance of $G$ result in more cumbersome calculations compared to the left invariant ones. We can pass from right invariant objects to left invariant ones by using the usual involution $i: g \mapsto g^{-1}$. In particular, right invariant vector fields are mapped to minus left invariant vector fields and a right invariant volume to a left invariant one. In the matrix representation~\eqref{eq:group} we have a basis of left invariant vector fields $a\p_a$, $a^2 \p_b$ which satisfy
$$
[a\p_a,a^2 \p_b] = 2 a^2 \p_b.
$$
Thus we can take $Z_1 = -a\p_a$, $Z_2 = -a^2 \p_b$ (here minus sign is a direct consequence of our a priori right invariant construction). The left invariant volume form is given by
$$
\mu_G = \frac{dadb}{a^3}.
$$
This way we arrive at the following coordinate representation of $\tilde P_q$:
\begin{equation}
\label{eq:locl_p}
\tilde P_q = (a\p_a)^2 + (a^2 \p_b)^2 - 1 - h_0.
\end{equation} 

We need to prove that $\tilde P_q$ is left invertible. To do this we first transform $\tilde P_q$ to a simpler form via some changes of variables. First we apply a partial Fourier transform with respect to the $b$ variable.
$$
\cF : u(a,b) \mapsto \hat{u}(a,\xi) = \int_\R u(a,b) e^{-i\xi b} db.
$$ 
This partial Fourier transform is a $L^2$-isometry
$$
\cF : L^2\left( G,\frac{da db}{a^3} \right)\to L^2\left( G,\frac{da d\xi}{a^3} \right) 
$$
and gives us
$$
\hat{P}_q = \left( a\p_a \right)^2 - a^4 \xi^2 - 1 - h_0.
$$
We can make a change of variables $x=a\sqrt{|\xi|}$ and obtain
\begin{equation}
\label{eq:operator}
\hat{P}_q = \left( x\p_x \right)^2 - x^4  - 1 - h_0.
\end{equation}
Note that this is a singular change variables, however the volume form is well defined because the singularity has measure zero. In this new coordinates the dual volume is equal to
$$
\hat{\mu}_G = \frac{|\xi| dx d\xi}{x^3}.
$$

Let us consider an operator of the same form as $\tilde P_q$ that we will denote as
$$
T= \left( x\p_x \right)^2 - x^4  - 1 - h_0
$$
acting on $C^\infty_c(\R_+)$. We want to extend its domain to a subspace $H\subset L^2(\R_+,dx/x^3)$ such that the operator $T: H \to  L^2(\R_+,dx/x^3)$ would be left invertible and $H$ would contain smooth compactly supported functions as a dense subset. After that we will use $H$ to prove left-invertibility of the operator $\tilde P_q$

To prove left invertibility $T$ we will use once again the CNQ conditions, by proving that $T$ is injective and left semi-Fredholm. We want to first represent $T$ as an operator coming from a differential operator compatible with a structure of a Lie manifold. We need to compactify $\R_+$. This can be done easily by local consideration around zero and infinity. The Lie manifold structure, which we denote by $\cV_X$, will be a $C^\infty$-module generated by a single vector field $X$ non-zero in $(0,+\infty)$ with certain asymptotics when $x\to 0$ and $x\to +\infty$.

From explicit form of $T$ we can see that it is compatible with the Lie manifold structure that is locally generated by $x\p_x$ around zero. In order to consider what happens at infinity, we make a change of variables $x\to y = 1/x$ and find that
$$
T = \left( y\p_y \right)^2 - \frac{1}{y^4}	-1 - h_0.
$$
After multiplying by $y^4$ we find an operator which is compatible with the Lie manifold structure that is locally generated by $y^3 \p_y$. Thus let 
$$
h: x\mapsto y = \frac{1}{x}
$$
and $X$ to be any vector field, such that
\begin{align*}
X&=x\p_x, & &0 < x < \varepsilon; \\
h_* X&=y^3\p_y, & &0 < y < \varepsilon.
\end{align*}
Then we choose $\cV_X$ to be the Lie manifold structure on $\R_+$ defined as a $C^\infty_b(\R_+)$-module generated by $X$, where $C^\infty_b(\R_+)$ are smooth bounded functions. This way we perform a two-point compactification of $\R_+$ by adding zero and infinity as boundaries. We denote this compactification by $\overline{\R}_+$

\begin{proposition}
Let $r: [0,+\infty)$ be a bounded smooth function such that $r(x) = x$ for $x\leq \varepsilon$ and $r(x) = 1$ for $x\geq 2\varepsilon$ and let $\tilde{r}(x)=r(1/x)$.
Then $T:(r\tilde{r}^2)H^2_{\cV_x}(\overline{\R}_+) \to L^2(\R_+,dx/x^3)$ is left invertible.
\end{proposition}

\begin{proof}
It is clear that $r$ and $\tilde{r}$ are defining functions for the zero boundary and boundary at infinity. We may choose a compatible volume form as
$$
\mu_{\cV_X} = \frac{r^2(x)dx}{x^3 \tilde{r}^4(x)} 
$$
and consider the corresponding Sobolev spaces $H^k_{\cV_X}(\overline{\R}_+)$.

First we prove that $T$ is injective. For this we need to solve $T u = 0$ and show that no solution lies in $(r\tilde{r}^2)H^2_{\cV_x}(\overline{\R}_+)$. Even more we will see that none of the solutions lies in $(r\tilde{r}^2)L^2_{\cV_x}(\overline{\R}_+)$. Indeed, there are two independent solutions to $Tu=0$ described via Bessel functions:
$$
u_-(x) = K_{\frac{\sqrt{1+h_0}}{2}}\left( \frac{x^2}{2} \right), \qquad 
u_+(x) = I_{\frac{\sqrt{1+h_0}}{2}}\left( \frac{x^2}{2} \right),
$$
and any other solution is a linear combination:
$$
u(x)= c_-u_-(x) + c_+u_+(x).
$$
However, neither $u_-$ nor $u_+$ are in $L^2(dx/x^3)$. This follows from the asymptotics of Bessel functions of the second kind. For $\nu>0$ the following asymptotic relations are valid
$\nu > 0$:
\begin{align*}
I_\nu(x) &\sim  \frac{1}{\Gamma(\nu+1)}\left(\frac{x}{2}\right)^\nu,\\
K_\nu(x) &\sim \frac{\Gamma(\nu)}{2}\left(\frac{2}{x}\right)^\nu
\end{align*} 
for $x \to 0+$ and
\begin{align*}
I_\nu(x) &\sim \frac{e^x x^{-1/2}}{\sqrt{2\pi}},\\
K_\nu(x) &\sim \sqrt{\frac{\pi}{2}}e^{-x} x^{-1/2}
\end{align*} 
for $x \to \infty$.

From here we see that $u_-$ and $u_+$ have different asymptotics close to zero and different asymptotics close to infinity. Close to zero functions from $r\tilde r^2 L^2_{\cV_X}(\overline{\R}_+)$ behave exactly like functions in $L^2(\R_+,dx/x^3)$. Function $u_+$ is never $dx/x^3$ square integrable for $x<\varepsilon$ small. On the other hand for $x>\varepsilon$ big $u_-$ is never locally square integrable because of the exponential growth. Hence a solution of $T u = 0$ can not be in $r\tilde r^2 L^2_{\cV_X}(\overline{\R}_+)$.

\bigskip

Now we need to prove that $T$ is left semi-Fredholm. This is equivalent to proving that the operator
$$
\tilde{T} = r^{-1}\tilde{r}^{2}T(r\tilde{r}^2): H^2_{\cV_X}(\overline{\R}_+) \to L^2_{\cV_X}(\overline{\R}_+) 
$$
is left semi-Fredholm. First note that $A_{\cV_X}$ is integrable to a Hausdorff simply connected Lie groupoid whose restriction to $(0,+\infty)$ is the pair groupoid. Indeed, $X$ is a complete vector field on $\overline{\R}_+$ and its flow defines the action of $\R$ on $\overline{\R}_+$. Hence we can take the integrating Lie groupoid to be the corresponding action groupoid and CNQ conditions can be applied.

We have to study invertibility of limit operators $\tilde{T}_0$ and $\tilde{T}_\infty$ which will be constant coefficient differential operators on $\R$. Let $z\in \R$ be the variable on $\R$. Recall that for $x=0$ in order to compute limit operators we have to replace $x\p_x$ with $\p_z$. After a lengthy computation we find that
$$
\tilde{T}_0 = \p_z^2 + 2\p_z - h_0.
$$
But this is exactly the operator~\eqref{eq:op_1d_ex} as in the 1D example from Section~\ref{sec:1D}. So we already know that this operator is left invertible.

Let us now compute the limit operator $\tilde{T}_\infty$. In this case we have to replace $y^3\p_y$ with $\p_z$. After a change of variables and another lengthy computation we obtain
$$
\tilde{T}_\infty = \p_z^2 - 1.
$$
This operator is also invertible on $\R$.

The left semi-Fredholm property now follows from Theorem~\ref{thm:cnq} and this finishes the proof that the operator $T$ is left invertible. 
\end{proof}

Next step is to extend the action of $T$ to the Fourier dual of $G$ such that it would remain to be left invertible. For this we will need some corollaries of left invertibility conditions of Lemma~\ref{lemma:inv} and Lemma~\ref{lemm:semibound}.

\begin{corollary}
\label{cor:1}
Let $A_i: H_i \to H_3$, $i=1,2$ be closed operators between the corresponding Hilbert spaces. Assume that there exist continuous inclusions $H_1 \hookrightarrow H_2 \hookrightarrow H_3$ such that $H_i$ is dense in $H_{i+1}$ and that $A_2$ extends $A_1$. Then $A_1$ is left invertible if, and only if, $A_2$ is left invertible. 
\end{corollary}

\begin{proof}

We can treat $H_i$ as subspaces of $H_3$. In order to prove the result by Lemma~\ref{lemm:semibound} it is enough to show that there exists a constant $c>0$, such that
$$
\|A_1 u\|_{H_3} \geq c\|u\|_{H_3} , \quad \forall u \in H_1 \quad \iff \quad \|A_2 v\|_{H_3} \geq c\|v\|_{H_3} , \quad \forall v \in H_2. 
$$
Since $H_1 \subset H_2$ the arrow in the left direction is true by restriction. The right arrow is true because the inclusions $H_1 \subset H_2 \subset H_3$ are continuous. Thus $A_2$ is the closure of $A_1$ in the $H_2$-norm and for any sequence $u_n \subset H_1$ converging to $v\in H_2$ in $H_2$, we have that $u_n \to v$ and $Au_n \to Av$ in $H_3$. 
\end{proof}

\begin{corollary}
\label{cor:2}
Let $A: D(A) \subset H_1 \to H_1$ be a left invertible operator in a Hilbert space $H_1$. Let $H_2$ be another Hilbert space. Then the operator $A \otimes \id: D(A)\otimes H_2 \subset H_1 \otimes H_2 \to H_1 \otimes H_2$. Is left invertible.
\end{corollary}

\begin{proof}
From Lemma~\ref{lemm:semibound} and definition of the tensor product it readily follows that
$$
\|(A \otimes \id) (u_1 \otimes u_2) \|_{H_1 \otimes H_2} \geq c \|u_1 \otimes u_2 \|_{H_1 \otimes H_2}, \qquad \forall u_1 \in D(A), u_2 \in H_2
$$
for some $c>0$. The rest follows by taking the closure and Lemma~\ref{lemma:inv}.
\end{proof}

We apply now Corollary~\ref{cor:2} to the operator $A = T\otimes \id$ from $L^2(\R_+,dx/x^3)\otimes L^2(\R,|\xi| d\xi)\simeq L^2(G,\hat{\mu}_G)$ to itself with domain $r\tilde r^2 H^2_{\cV_x}(\overline{\R_+})\otimes L^2(\R,|\xi| d\xi)$ and find that this operator is left invertible. We can also consider $A$ as an operator
\begin{equation}
\label{eq:A_def}
A: H^2_{x\p_x}\left( \R_+,\frac{dx}{x^3} \right)\otimes L^2\left(\R_+, |\xi|d\xi\right) \to L^2\left( \R_+,\frac{dx}{x^3} \right)\otimes L^2\left(\R_+, |\xi|d\xi\right),
\end{equation}
where, as the notation suggests, the Sobolev space on the left is the closure of $C^\infty_c(\R_+)$ in the norm
$$
\|u\|^2_{H^2_{x\p_x}} = \int_{\R_+} \left( |u|^2 + |x\p_x u|^2 + |(x\p_x)^2 u|^2 \right)\frac{dx}{x^3}.
$$
If we can prove that $A$ is left invertible on this new domain, then we can apply Corollary~\ref{cor:1} to prove that $A$ is left invertible on $H^2(G)$. To see this let us write down the norm for the Hilbert space $H^2(G)$ in the $(x,\xi)$ coordinates. Using the partial Fourier transform $\cF$ and change of variables $(a,\xi)\mapsto (x,\xi)$ we can write it as 
$$
\int_\R \left( \int_{\R_+} \left( |u|^2 + |x\p_x u|^2 + |(x\p_x)^2u|^2 + |x^2 u|^2 + |x^3 \p_x u|^2+|x\p_x(x^2 u)|^2+ |x^4 u|^2 \right)\frac{dx}{x^3}\right)|\xi| d\xi 
$$
We can now see from the explicit forms of the norms that $H^2_{x\p_x}\left( \R_+,\frac{dx}{x^3} \right)\otimes L^2\left(\R_+, |\xi|d\xi\right)$ continuously embeds in $H^2(G,\hat{\mu})$. Thus left invertibility of $A$ on $H^2(G)$ will follow from Corollary~\ref{cor:1}.

This means that in order to finish the proof of Proposition~\ref{prop:grushin_suffering} and the main Theorem~\ref{thm:main}, we only have to prove that the operator $A$ defined in~\eqref{eq:A_def} is left invertible.

\begin{lemma}
Operator $A$ defined in~\eqref{eq:A_def} is left invertible. 
\end{lemma}

\begin{proof}
We have seen that $A$ as an operator
$$
A:r\tilde r^2 H^2_{\cV_x}(\overline{\R}_+)\otimes L^2(\R,|\xi| d\xi) \to L^2(G,\hat{\mu})
$$
is left invertible. Note also that $r\tilde r^2 H^2_{\cV_x}(\overline{\R}_+)\otimes L^2(\R,|\xi| d\xi)$ embeds continuously into $L^2(G,\hat{\mu})$ by Lemma~\ref{lemma:emb} and definition of the tensor product. So it only remains to prove that $r\tilde r^2 H^2_{\cV_x}(\overline{\R}_+)$ embeds continuously into $H^2_{x\p_x}\left( \R_+,dx/x^3 \right)$, i.e., that there exists a constant $c>0$ such that
\begin{equation}
\label{eq:norm_ineq}
\|u\|^2_{H^2_{x\p_x}}\leq c\|u\|^2_{r\tilde r^2 H^2_{\cV_X}(\overline{\R}_+)}, \qquad \forall u \in r\tilde r^2 H^2_{\cV_X}(\overline{\R}_+).
\end{equation}
We can write the $r\tilde r^2 H^2_{\cV_X}(\overline{\R}_+)$ norm as 
$$
\|u\|^2_{r\tilde r^2 H^2_{\cV_X}}  = \int_{\R_+} \left( \left|\frac{u}{r\tilde{r}^2}\right|^2 + \left|x\p_x \left( \frac{u}{r\tilde{r}^2} \right)\right|^2 + \left|(x\p_x)^2 \left( \frac{u}{r\tilde{r}^2} \right)\right|^2 \right)\frac{r^2 dx}{x^3\tilde{r}^4}. 
$$
We will prove~\eqref{eq:norm_ineq} for smooth function with compact support and the rest will follow by completion. The proof is straightforward and is given mainly for completeness.

Let $u = r\tilde{r}^2 v$. Then we have
\begin{align*}
\|r\tilde{r}^2 v\|^2_{H^2_{x\p_x}(\overline{\R}_+)} &= \int_{\R_+} |r\tilde r^2 v|^2\frac{dx}{x^3} + \int_{\R_+} |x\p_x \left( r\tilde r^2 v \right)|^2\frac{dx}{x^3}  + \\
&+\int_{\R_+}|(x\p_x)^2 \left( r\tilde r^2 v \right)|^2 \frac{dx}{x^3} = I_0 + I_1 + I_2.  
\end{align*}
Recalling that $\tilde{r}(x)\in[0,1] $ we obtain
$$
I_0 = \int_{\R_+} \tilde r^8 |v|^2\mu_X\leq \int_{\R_+}|v|^2\mu_X.
$$
Terms $I_1$ and $I_2$ are proved similarly. We only show the argument for $I_1$ since for $I_2$ it is almost identical but with more steps.
We have by the chain rule and Young inequality
\begin{align*}
I_1 &= \|x\p_x(r\tilde{r}^2)v \|^2_{L^2(\R_+,dx/x^3)}+\langle x\p_x(r\tilde{r}^2)v, r\tilde{r}^2 x\p_xu\rangle_{L^2(\R_+,dx/x^3)} + \\
&+  \|r\tilde{r}^2 x\p_xu \|^2_{L^2(\R_+,dx/x^3)} \leq 2\|x\p_x(r\tilde{r}^2)v \|^2_{L^2(\R_+,dx/x^3)}+ 2\|r\tilde{r}^2 x\p_xu \|^2_{L^2(\R_+,dx/x^3)},
\end{align*}
Now for the first term we find
$$
\|x\p_x(r\tilde{r}^2)v \|^2_{L^2(\R_+,dx/x^3)} = \int_{R_+} (x\p_x(r\tilde{r}^2))^2 |v|^2 \frac{\tilde{r}^4}{r^2} \mu_X. 
$$
Function
$$
f(x) = \frac{x\p_x(r(x)\tilde{r}^2(x))^2 \tilde{r}^4(x)}{r^2(x)}
$$
is smooth and bounded on $\R_+$. Indeed, $r,\tilde{r}$ are smooth and bounded. So the only problem can occur close to $x=0$. Recall that close to zero $r = x$ and $\tilde{r} = 1$ by construction. So $f(x) \equiv 1$ close to $x=0$. Hence
$$
\|x\p_x(r\tilde{r}^2)v \|^2_{L^2(\R_+,dx/x^3)}\leq C\int_{R_+} |v|^2 \mu_X.
$$
The other term is handled exactly as $I_0$ and $I_2$ is bounded by a similar argument. The lemma now follows by completion.
\end{proof}

This proves Proposition~\ref{prop:grushin_suffering} and finishes the proof of Theorem~\ref{thm:main}.

\section{Conclusions, final remarks, extensions}
\label{sec:final}

Theorem~\ref{thm:main} illustrates that it is possible to treat various singularities in almost-Riemannian geometry from a unified perspective. This raises the question of what is the domain of applicability and where are the practical limits of this method for obtaining information about geometric operators on sub-Riemannian manifolds and more generally about geometric differential operators on singular spaces.

First of all we note that it is possible to prove an analogue of Proposition~\ref{prop:prop2} in the case when there are no tangency points for the Laplace operator $\Delta$ itself. For this it is enough to consider instead of $\tilde P$ operators $\tilde P^\gamma$ defined as
$$
\tilde P^\gamma =s^{2-\gamma}\Delta s^{\gamma},
$$
for $\gamma = 1\pm \varepsilon$ and repeat the proof. But it would have resulted in much longer formulas and a more convoluted analysis. Theorem~\ref{thm:main} is already enough to explain the concept. In the case of the tangency points weights $s^\gamma$ are not enough. They behave similarly to cusp points and irregular singular points which require non-polynomial weights.

Secondly we should note that the method applies to non-generic structures as well, for examples manifold analogues of $\alpha$-Grushin planes locally modelled by vector fields $\p_x,x^n e^{\phi(x,y)}\p_y$. Indeed, the construction of the integrating Lie groupoid would be very similar and the analysis as well. Recall that in the dimension two there are only two simply connected Lie groups: the Euclidean space and the group of affine transformations of the real line. In both cases their harmonic analysis is relatively simple. For this reason we did not need to use the machinery of non-commutative harmonic analysis in this article. Partial Euclidean Fourier transform was enough.

One can then ask how the method would work in higher dimensions. Two problems are encountered in this case. The first obstacle is that we need to prove the existence of a Hausdorff integrating Lie groupoid. Even though there are no obstructions to Hausdorfness in dimension two, one can imagine that there might be some in higher dimensions. The second obstacle is that the left invertibility of limit operators might be difficult to prove in practice. From this point of view the best structures would be those that have a relatively simple harmonic analysis associated to the isotropy groups at the singular set. Having small dimensional coadjoint orbits would certainly help. In particular, if orbits are two-dimensional, then the non-commutative Fourier transform would transform limit operators to one-dimensional operators. All Lie groups of dimensions three and smaller have this property. Lie groups with coadjoint orbits of dimensions less or equal than two were classified in~\cite{coadj_2}. One of the consequences of this classification is that in dimensions greater than six all Lie groups having two-dimensional coadjoint orbits are semidirect products of abelian groups. This partially explains the prevalence in the literature of asymptotically Euclidean and asymptotically Hyperbolic manifolds which have isotropy groups of this type.

One can also use the algorithm outlined in the beginning of the Section~\ref{sec:1D} to already known results. For example, it is possible to recover the closure results for wedge operators from~\cite{mendoza}. The advantage of the method presented in the article is that CNQ Fredholm conditions can be used as a black box. There is no need for constructing an associated pseudo-differential calculus by hand. However going into the PDO structure of the problem allows to prove deeper results. For example, in~\cite{mendoza} as a by-product of constructing the parametrix using a hand-made PDO calculus the authors also proved asymptotics of the trace of the resolvent of elliptic wedge operators.
 
Finally it would be interesting to generalise this kind of results to other classes of singular sub-Riemannian manifolds. For example, an analogue of the main Theorem~\ref{thm:main} for generic Martinet manifolds~\cite{Trelat} might be proven using a similar strategy modulo some modifications required due to hypoellipticity. But of course, it would be good to have a general theoretical basis for any sub-Riemannian structure regardless of dimensions or singularities. 

\bigskip

\textbf{Acknowledgements:} 
This work was supported by the French ANR project Quaco ANR-17-CE40-0007-01 and through the CIDMA Center for Research and Development in Mathematics and Applications, and the Portuguese Foundation for Science and Technology (``FCT - Funda\c{c}\~ao para a Ci\^encia e a Tecnologia'') within the project UIDB/04106/2020.

The author would like to thank Victor Nistor for his patient explanations regarding Lie groupoids and results of~\cite{nistor_fred} as well as Eugenio Pozzoli for many fruitful discussions and useful remarks.

\bibliographystyle{plain}
\bibliography{references}

\end{document}